\documentclass[11pt,reqno,oneside]{amsart}

\usepackage{graphicx}
\usepackage{amssymb}
\usepackage{epstopdf}

\usepackage[a4paper, total={6in, 9in}]{geometry}

\usepackage{amsmath,amsfonts,amsthm,mathrsfs,amssymb,cite}
\usepackage[usenames]{color}

\newtheorem{thm}{Theorem}[section]

\newtheorem{lem}{Lemma}[section]

\theoremstyle{definition}

\theoremstyle{remark}

\newtheorem{rem}{Remark}[section]
\numberwithin{equation}{section}

\numberwithin{equation}{section}

\newcounter{saveeqn}

\newcommand{\eqnref}[1]{(\ref {#1})}

\newcommand{\Bz}{\mathbf{z}}

\newcommand{\bE}{\mathbf{E}}

\newcommand{\bH}{\mathbf{H}}

\newcommand{\Bx}{\mathbf{x}}
\newcommand{\By}{\mathbf{y}}

\newcommand{\bZ}{\mathbf{z}}

\newcommand{\Gl}{\lambda}

\newcommand{\Gs}{\sigma}

\newcommand{\Ge}{\varepsilon}

\newcommand{\tdx}{\tilde{\Bx}}
\newcommand{\tdy}{\tilde{\By}}

\newcommand{\Acal}{\mathcal{A}}
\newcommand{\Kcal}{\mathcal{K}}
\newcommand{\Lcal}{\mathcal{L}}
\newcommand{\Scal}{\mathcal{S}}

\newcommand{\Mcal}{\mathcal{M}}

\newcommand{\Ocal}{\mathcal{O}}

\newcommand{\ds}{\displaystyle}

\newcommand{\RR}{\mathbb{R}}

\newcommand{\p}{\partial}

\newcommand{\beq}{\begin{equation}}
\newcommand{\eeq}{\end{equation}}

\DeclareMathAlphabet{\itbf}{OML}{cmm}{b}{it}

\title[Identifying magnetized anomalies in geomagnetism]{On identifying magnetized anomalies using geomagnetic monitoring}

\author{Youjun Deng}
\address{School of Mathematics and Statistics, Central South University, Changsha, Hunan, China.}
\email{youjundeng@csu.edu.cn, dengyijun\_001@163.com}

\author{Jinhong Li}
\address{School of Science, Qilu University of Technology (Shandong Academy of Sciences), Jinan, Shandong, China}
\email{lijinhong@qlu.edu.cn}

\author{Hongyu Liu}
\address{Department of Mathematics, Hong Kong Baptist University, Kowloon, Hong Kong SAR, China}
\email{hongyu.liuip@gmail.com, hongyuliu@hkbu.edu.hk}

\date{} 
\begin{document}
\maketitle

\begin{abstract}

We propose and investigate the inverse problem of identifying magnetized anomalies beneath the Earth using the geomagnetic monitoring. Suppose a collection of magnetized anomalies presented in the shell of the Earth. The presence of the anomalies interrupts the magnetic field of the Earth, monitored above the Earth. Using the difference of the magnetic fields before and after the presence of the magnetized anomalies, we show that one can uniquely recover the locations as well as their material parameters of the anomalies. Our study provides a rigorous mathematical theory to the geomagnetic detection technology that has been used in practice.

\medskip



\end{abstract}

\section{Introduction}

It is widely known that the Earth as well as most of the planets in the solar system all generate magnetic fields through the motion of electrically conducting fluids \cite{Fey,Weiss}. Earth's magnetic field, also known as the {\it geomagnetic} field, is the magnetic field that extends from the Earth's interior out into the space. Following the general discussion in \cite{BPCo96}, the geomagnetic field is described by the Maxwell system as follows.

Let the Earth be of a core-shell structure with $\Sigma_c$ and $\Sigma$, respectively, signifying the core and the Earth. It is assumed that both $\Sigma_c$ and $\Sigma$ are bounded simply-connected $C^2$ domains in $\mathbb{R}^3$ and $\Sigma_c\Subset\Sigma$. $\Sigma_s:=\Sigma\backslash\overline{\Sigma_c}$ signifies the shell of the Earth. We note that in the literature, it is usually assumed that the Earth and its core are concentric balls of radii $R_1$ and $R_0$ with $R_0\ll R_1$. However, we shall not impose such a restrictive assumption in our study. Let $\varepsilon, \mu$ and $\sigma$ be all real-valued $L^\infty$ functions, such that $\varepsilon$ and $\mu$ are positive and $\sigma$ is nonnegative.
The functions $\varepsilon$, $\mu$ and $\sigma$
signify
the electromagnetic (EM) medium parameters in $\mathbb{R}^3$, and are referred to as the electric permittivity,
 the magnetic permeability and the electric conductivity, respectively.
Let $\varepsilon_0$ and $\mu_0$ denote, respectively, the permittivity and the permeability of the uniformly homogeneous free space $\mathbb{R}^3\backslash\overline{\Sigma}$. The material distribution is described by
\beq\label{eq:paradef01}
\begin{split}
\Gs(\Bx)=\Gs_c(\Bx)\chi(\Sigma_c),\ \ \quad\mu(\Bx)=(\mu_c(\Bx)-\mu_0)\chi(\Sigma_c)+\mu_0,\\
\varepsilon(\Bx)=(\varepsilon_c(\Bx)-\varepsilon_0)\chi(\Sigma_c)+(\varepsilon_s(\Bx)-\varepsilon_0)\chi(\Sigma_s)+\varepsilon_0,\quad
\end{split}
\eeq
where and also in what follows, $\chi$ denotes the characteristic function. By \eqref{eq:paradef01}, we know that the mediums in the core and shell of the Earth are respectively characterized by $(\Sigma_c; \varepsilon_c,\mu_c,\sigma_c)$ and $(\Sigma_s; \varepsilon_s, \mu_0)$. Let $\mathcal{E}(\Bx,t)$ and $\mathcal{H}(\Bx, t)$, $(\Bx, t)\in\mathbb{R}^3\times\mathbb{R}_+$, respectively, denote the electric and magnetic fields of the Earth. They satisfy the following Maxwell system for $ (\Bx,t)\in \RR^3\times\RR_+$ (cf. \cite{BPCo96})
 \begin{equation}\label{eq:pss}
\begin{cases}
\nabla\times\mathcal{H}(\Bx,t)=\varepsilon(\Bx)\p_t \mathcal{E}(\Bx,t)+\Gs(\Bx)(\mathcal{E}(\Bx,t)+\mu(\Bx)\mathbf{v}\times\mathcal{H}(\Bx,t)),\medskip \\
 \nabla\times\mathcal{E}(\Bx,t)=-\mu(\Bx)\p_t \mathcal{H}(\Bx,t), &\medskip \\
\nabla\cdot (\mu(\Bx)\mathcal{H}(\Bx,t))=0,\quad \nabla\cdot (\varepsilon(\Bx)\mathcal{E}(\Bx,t))=\rho(\Bx, t), &\medskip \\
\mathcal{E}(\mathbf{x}, 0)=\mathcal{H}(\mathbf{x}, 0)=0,&
\end{cases}
 \end{equation}
 where $\rho(\Bx, t)=\rho_c(\Bx, t)\chi(\Sigma_c)\in H^1(\mathbb{R}_+, L^2(\Sigma_c))$ stands for the charge density of the Earth core, and $\mathbf{v}\in L^\infty(\Sigma)$ is the fluid velocity of the Earth. In \eqref{eq:pss}, $\mathbf{v}\times\mu\mathcal{H}$ is the so-called motional electromotive force generated by the rotation of the Earth.

Next we suppose that a collection of magnetized anomalies presented in the shell of the Earth. Let $D_l$, $l=1,2,\ldots, l_0$, denote the magnetized anomalies, where $D_l$, $1\leq l\leq l_0$ is a simply-connected Lipschitz domain such that
\begin{figure}
\begin{center}
  \includegraphics[width=2.3in,height=1.8in]{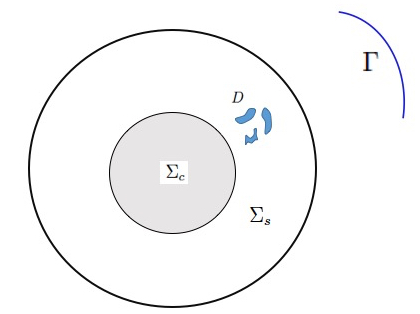}
  \end{center}
  \caption{Schematic illustration of identifying magnetized anomalies using the geomagnetic monitoring. \label{fig1}}
\end{figure}
the corresponding material parameters are given by $\varepsilon_l, \mu_l$ and $\sigma_l$. It is assumed that $\varepsilon_l, \mu_l$ and $\sigma_l$ are all positive constants with $\mu_l\neq \mu_0$, $1\leq l\leq l_0$. With the presence of the magnetized anomalies $(D_l; \varepsilon_l, \mu_l, \sigma_l)$, $l=1,2,\ldots,l_0$, in the shell of the Earth, the EM medium configuration in the space $\mathbb{R}^3$ is then described by
\beq\label{eq:paradef02}
\begin{split}
\Gs(\Bx)=& \Gs_c(\Bx)\chi(\Sigma_c)+\sum_{l=1}^{l_0} \sigma_l\chi(D_l),\\
\quad\mu(\Bx)=& (\mu_c(\Bx)-\mu_0)\chi(\Sigma_c)+\sum_{l=1}^{l_0} (\mu_l-\mu_0)\chi(D_l)+\mu_0,\\
\varepsilon(\Bx)=& (\varepsilon_c(\Bx)-\varepsilon_0)\chi(\Sigma_c)+(\varepsilon_s(\Bx)-\varepsilon_0)\chi(\Sigma_s\setminus\overline{\bigcup_{l=1}^{l_0}D_l})+\sum_{l=1}^{l_0} (\varepsilon_l-\varepsilon_0)\chi(D_l)+\varepsilon_0.
\end{split}
\eeq
In the sequel, we let $(\mathcal{E}_0, \mathcal{H}_0)$ be the solution to the Maxwell system \eqref{eq:pss} associated with the medium configuration in \eqref{eq:paradef01}, and $(\mathcal{E}, \mathcal{H})$ be the solution to \eqref{eq:pss} associated with \eqref{eq:paradef02}. Let $\Gamma$ be an open surface located away from $\Sigma$. In the current article, we are mainly concerned with the following inverse problem,
\begin{equation}\label{eq:geom1}
\left(\mathcal{H}(\Bx, t)-\mathcal{H}_0(\Bx, t) \right)\bigg|_{(\Bx,t)\in\Gamma\times\mathbb{R}_+}\longrightarrow \bigcup_{l=1}^{l_0} (D_l; \varepsilon_l,\mu_l,\sigma_l).
\end{equation}
That is, one intends to recover the magnetized anomalies by monitoring the change of the geomagnetic field away from the Earth. It is emphasized that in \eqref{eq:geom1}, we do not assume that the medium configuration of the Earth, and the charge density and fluid velocity of the Earth core are known a priori. From a practical point view, the only thing known before the presence of the magnetized anomalies is the monitored geomagnetic field $\mathcal{H}_0$ on $\Gamma$. We would also like to emphasize that in practice, $\mathbb{R}_+$ in \eqref{eq:geom1} can actually be replaced by a finite time interval, and we shall remark this point in Section~\ref{sect:5}. The magnetic anomaly detecting (MAD) technique has been used in various practical applications. The magnetometer that can measure minute variations in the Earth's magnetic field has been used by military forces to detect submarines. The military MAD equipment is a descendent of geomagnetic survey or aeromagnetic survey instruments used to search for minerals by detecting their disturbance of the normal earth-field. The aim of this study is to provide a rigorous mathematical theory for this important applied technology. Indeed, we establish global uniqueness results for the nonlinear inverse problem \eqref{eq:geom1} in certain practically important and generic scenarios.

The rest of the section is devoted to recasting the time-dependent inverse problem \eqref{eq:geom1} to its counterpart in the frequency domain via the Fourier transform approach. We refer to \cite{Leis1,Leis2} for the well-posedness of the forward Maxwell system \eqref{eq:pss}, and in particular the unique existence of a pair of solutions $(\mathcal{E}, \mathcal{H})\in H^1(\mathbb{R}_+^0, H_{loc}(\mathrm{curl},\mathbb{R}^3))^2$. In the sequel, we shall make use of the following temporal Fourier transform for $(\Bx, \omega)\in\mathbb{R}^3\times\mathbb{R}_+$,
\begin{equation}\label{eq:cond1}
\mathbf{J}(\Bx,\omega)=\mathcal{F}_t(\mathcal{J}):=\frac{1}{2\pi}\int_0^\infty \mathcal{J}(\Bx, t) e^{\mathrm{i}\omega t}\ dt, \quad\mathcal{J}=\mathcal{E}\ \ \mbox{or}\ \ \mathcal{H},
\end{equation}
such that $\mathbf{J}(\Bx, \omega)\in H_{loc}(\mathrm{curl},\mathbb{R}^3)$.
Set
\begin{equation}\label{eq:ft2}
\mathbf{E}=\mathcal{F}_t(\mathcal{E}),\ \mathbf{H}=\mathcal{F}_t(\mathcal{H}),\ \mathbf{E}_0=\mathcal{F}_t(\mathcal{E}_0),\ \mathbf{H}_0=\mathcal{F}_t(\mathcal{H}_0),\ \hat{\rho}=\mathcal{F}_t(\rho).
\end{equation}
Throughout the paper, under a certain generic causality condition on the geomagnetic configuration of the Earth, we assume that the Fourier transforms in \eqref{eq:ft2} all exist. In order to appeal for a general inverse problem study, we shall not explore this point in the current article; see also our remark concerning the geomagnetic configuration of the Earth after \eqref{eq:geom1}. With the above assumption,
the time-dependent Maxwell system \eqref{eq:pss} is then reduced to the following time-harmonic system in the frequency domain,
 \begin{equation}\label{eq:pss01}
\left \{
 \begin{array}{ll}
\nabla\times\bH=-\mathrm{i}\omega\varepsilon\bE+\Gs(\bE+\mu\mathbf{v}\times\mathbf{H})  & \mbox{in} \ \ \RR^3,\medskip\\
\nabla\times\bE=\mathrm{i}\omega\mu\bH & \mbox{in} \ \ \RR^3,\medskip\\
\nabla\cdot (\mu\mathbf{H})=0,\quad \nabla\cdot (\varepsilon\bE)=\hat\rho  & \mbox{in} \ \ \RR^3,\medskip\\
 \lim_{\|\Bx\|\rightarrow\infty} \|\Bx\| \big( \sqrt{\mu_0}\bH \times\hat{\Bx}- \sqrt{\varepsilon_0}\bE\big)=0,
 \end{array}
 \right .
 \end{equation}
 where the last limit is the Silver-M\"uller radiation condition and it holds uniformly in all directions $\hat{\Bx}:=\Bx/\|\Bx\|\in\mathbb{S}^{2}$. The Silver-M\"uller radiation condition characterizes the outward radiating waves (cf. \cite{Leis2}). That is, in order to establish the recovery results for the inverse problem \eqref{eq:geom1}, from a practical point of view, we shall only make use of the measurement data from the outward radiating EM waves.
We refer to \cite{LRX} for the related study on the unique existence of $(\bE,\bH)\in H_{loc}(\mathrm{curl}, \mathbb{R}^3)^2$ to \eqref{eq:pss01}. In particular, we know that there holds the following asymptotic expansion as $\|\Bx\|\rightarrow+\infty$ (cf. \cite{CK,Ned}),
\begin{equation}\label{eq:farf1}
\mathbf{H}(\Bx)=\frac{e^{\mathrm{i}k_0\|\Bx\|}}{\|\Bx\|} \mathbf{H}^\infty(\hat{\Bx})+\mathcal{O}\left(\frac{1}{\|\Bx\|^2} \right),\quad k_0:=\omega\sqrt{\varepsilon_0\mu_0}.
\end{equation}
In a similar manner, one can derive the Maxwell system for the EM fields $\mathbf{E}_0$ and $\mathbf{H}_0$, as well as the corresponding magnetic far-field pattern $\mathbf{H}_0^\infty$. The inverse problem \eqref{eq:geom1} can then be recast as
\begin{equation}\label{eq:geom2}
\mathbf{H}^\infty(\hat{\Bx};\omega)-\mathbf{H}_0^\infty(\hat{\Bx};\omega)\big|_{(\hat{\Bx},\omega)\in \hat{\Gamma}\times\mathbb{R}_+}\longrightarrow \bigcup_{l=1}^{l_0} (D_l; \varepsilon_l,\mu_l,\sigma_l),\ \ \hat{\Gamma}:=\{\hat{\Bx}\in\mathbb{S}^2; \hat{\Bx}=\frac{\Bx}{\|\Bx\|},\ \Bx\in\Gamma\}\subset\mathbb{S}^2.
\end{equation}
On the other hand, by the real-analyticity of $\mathbf{H}^\infty$ and $\mathbf{H}_0^\infty$, together with the Rellich theorem (cf. \cite{CK}), we know that the inverse problem \eqref{eq:geom2} is equivalent to the following one,
\begin{equation}\label{eq:geom3}
\mathbf{H}(\Bx;\omega)-\mathbf{H}_0(\Bx;\omega)\big|_{(\Bx,\omega)\in\partial\Sigma\times\mathbb{R}_+}\longrightarrow \bigcup_{l=1}^{l_0} (D_l; \varepsilon_l,\mu_l,\sigma_l).
\end{equation}
We are mainly concerned with the theoretical unique recovery results for the aforementioned inverse problem \eqref{eq:geom2}, or equivalently \eqref{eq:geom3}. It is remarked that in our subsequent study of \eqref{eq:geom2} or \eqref{eq:geom3}, we actually make use of the low-frequency asymptotics of the geomagnetic fields. That is, in \eqref{eq:geom2} and \eqref{eq:geom3}, it is sufficient for us to know the geomagnetic fields with frequencies from a neighbourhood of the zero frequency.

The rest of the paper is organized as follows. In Section 2, we derive the asymptotic low-frequency approximation of the background magnetic field $\bH_0$ and show that the leading-order term is conservative with an explicit form of the corresponding potential function. Section 3 is devoted to the two-level asymptotic approximations of the perturbed magnetic field $\bH$. One approximation is derived in terms of the frequency and the other one is derived in terms of the size of anomalies. Finally, in Section 4, we establish the unique recovery results on identifying the positions as well as the material properties of the magnetized anomalies.

\section{Integral representation and asymptotics of $\bH_0$}
In this section, we present the integral representation of the magnetic field generated by the Earth core. We are mainly concerned with the magnetic filed distribution outside the Earth core, namely, $\bH_0$ in $\RR^3\setminus\overline{\Sigma_c}$. Before proceeding, we present some preliminary knowledge on layer potential techniques (cf. \cite{HK07:book, Ned}).

\subsection{Layer potentials}
Let $ \Gamma_k$ be the fundamental solution to the PDO $(\Delta+k^2)$, that is given by
\begin{equation}\label{Gk} \ds \Gamma_k
(\Bx) = -\frac{e^{ik\|\Bx\|}}{4 \pi \|\Bx\|},\ \ \Bx\in\mathbb{R}^3\ \ \mbox{and}\ \ \Bx\neq \mathbf{0}.
 \end{equation}
For any bounded domain $B\subset \RR^3$, we denote by $\Scal_B^k: H^{-1/2}(\p B)\rightarrow H^{1}(\RR^3\setminus\p B)$ the single layer potential operator given by
\beq\label{eq:layperpt1}
\Scal_{B}^{k}[\phi](\Bx):=\int_{\p B}\Gamma_k(\Bx-\By)\phi(\By)d s_\By,
\eeq
and $\Kcal_B^{k}: H^{1/2}(\p B)\rightarrow H^{1/2}(\p B)$ the Neumann-Poincar\'e operator
\beq\label{eq:layperpt2}
\Kcal_{B}^{k}[\phi](\Bx):=\mbox{p.v.}\quad\int_{\p B}\frac{\p\Gamma_k(\Bx-\By)}{\p \nu_y}\phi(\By)d s_\By,
\eeq
where p.v. stands for the Cauchy principle value. In \eqref{eq:layperpt2} and also in what follows, unless otherwise specified, $\nu$ signifies the exterior unit normal vector to the boundary of the concerned domain.
It is known that the single layer potential operator $\Scal_B^k$ satisfies the following trace formula
\beq \label{eq:trace}
\frac{\p}{\p\nu}\Scal_B^{k}[\phi] \Big|_{\pm} = (\pm \frac{1}{2}I+
(\Kcal_{B}^{k})^*)[\phi] \quad \mbox{on} \, \p B, \eeq
where $(\Kcal_{B}^{k})^*$ is the adjoint operator of $\Kcal_B^{k}$.
In addition, for a density $\Phi \in \mathrm{TH}(\mbox{div}, \p B)$, we define the
vectorial single layer potential by
\beq\label{defA}
\ds\mathcal{A}_B^{k}[\Phi](\Bx) := \int_{\p B} \Gamma_{k}(\Bx-\By)
\Phi(\By) d s_\By, \quad \Bx \in \RR^3\setminus\p B.
\eeq
It is known that $\nabla\times\mathcal{A}_B^{k}$ satisfies the following jump formula
\begin{equation}\label{jumpM}
\nu \times \nabla \times \mathcal{A}_B^{k}[\Phi]\big\vert_\pm = \mp \frac{\Phi}{2} + \mathcal{M}_B^{k}[\Phi] \quad \mbox{on}\,  \p B,
\end{equation}
where \begin{equation*}
\forall \Bx\in \p B, \quad \nu \times \nabla \times \mathcal{A}_B^{k}[\Phi]\big\vert_\pm (\Bx)= \lim_{t\rightarrow 0^+} \nu \times \nabla \times \mathcal{A}_B^{k}[\Phi] (\Bx\pm t \nu)
\end{equation*}
and
\beq\label{Mk}
\mathcal{M}^{k}_B[\Phi](\Bx)= \mbox{p.v.}\quad\nu  \times \nabla \times \int_{\p B} \Gamma_{k}(\Bx-\By) \Phi(\By) d s_\By.
\eeq
We also define $\mathcal{L}^k_B: \mathrm{TH}({\rm div}, \p B) \rightarrow \mathrm{TH}({\rm div}, \p B)$ by
\beq\label{Lk}
\mathcal{L}^k_B[\Phi](\Bx):= \nu_\Bx  \times \nabla\times\nabla\times \mathcal{A}_B^k[\Phi](\Bx)=\nu_\Bx  \times \big(k^2\mathcal{A}_B^k[\Phi](\Bx)
+\nabla\nabla\cdot\mathcal{A}_B^k[\Phi](\Bx)\big),
\eeq
where we have made use of the formula $\nabla\times\nabla\times=-\Delta+\nabla\nabla\cdot$, which shall be frequently used in the sequel.

Next we introduce some function spaces on the boundary surface for the subsequent use.
Let  $\nabla_{\p B}\cdot$ denote the surface divergence. Denote by $L_T^2(\p B):=\{\Phi\in {L^2(\p B)}^3, \nu\cdot \Phi=0\}$. Let $H^s(\partial B)$ be the usual Sobolev space of order $s\in\mathbb{R}$ on $\partial B$. Set
\begin{align*}
\mathrm{TH}({\rm div}, \p B):&=\Bigr\{ {\Phi} \in L_T^2(\partial B):
\nabla_{\partial B}\cdot {\Phi} \in L^2(\partial B) \Bigr\},\\
\mathrm{TH}({\rm curl}, \p B):&=\Bigr\{ {\Phi} \in L_T^2(\partial B):
\nabla_{\partial B}\cdot ({\Phi}\times {\nu}) \in L^2(\partial B) \Bigr\},
\end{align*}
endowed with the norms
\begin{align*}
&\|{\Phi}\|_{\mathrm{TH}({\rm div}, \p B)}=\|{\Phi}\|_{L^2(\p B)}+\|\nabla_{\p B}\cdot {\Phi}\|_{L^2(\p B)}, \\
&\|{\Phi}\|_{\mathrm{TH}({\rm curl}, \p B)}=\|{\Phi}\|_{L^2(\p B)}+\|\nabla_{\p B}\cdot({\Phi}\times \nu)\|_{L^2(\p B)}.
\end{align*}
{We mention that $\frac{I}{2}\pm \Mcal_{B}^{k}$ is invertible on $\rm{TH}({\rm div}, \p B)$} when $k$ is sufficiently small (see e.g., \cite{ADM14,T}). In the following, if $k=0$, we formally set $\Gamma_{k}$ introduced in \eqref{Gk} to be $\Gamma_0=-1/(4\pi\|\Bx\|)$, and the other integral operators introduced above can also be formally defined when $k=0$. Finally, we define $k_s:=\omega\sqrt{\mu_0\varepsilon_s}$.

\subsection{Integral representation and approximation}
Let $(\bE_0, \bH_0)$ be the solution to \eqnref{eq:paradef01} and \eqnref{eq:pss01}. In this section, we consider the steady fields generated by the Earth's core for $\omega\in [0, \omega_0)$, with $\omega_0>0$ a fixed and sufficiently small real number. For the subsequent use of the inverse problem study, we shall derive the integral representation and the low-frequency asymptotics of the fields.

By using the transmission conditions across $\p \Sigma_{s}$, $(\bE_0, \bH_0)$ is the solution to the following transmission problem
\beq\label{eq:pss02}
\left \{
 \begin{array}{ll}
\displaystyle{\nabla\times\bH_0=-i\omega\varepsilon_c\bE_0+\Gs_c(\bE+\mu_c\mathbf{v}\times\bH_0)},  & \mbox{in} \ \ \Sigma_c\medskip\\
\displaystyle{\nabla\times\bE_0=i\omega\mu_c\bH_0,} & \mbox{in} \ \ \Sigma_c,\medskip \\
\displaystyle{\nabla\cdot (\mu_c\bH_0)=0,\quad \nabla\cdot (\varepsilon_c\bE_0)=\hat\rho}, &\mbox{in} \ \ \Sigma_c,\medskip\\
\displaystyle{{[}\nu\times\bE_0]={[}\nu\times\bH_0]=0,} & \mbox{on} \ \ \p \Sigma_s,\medskip\\
 \displaystyle{\nabla\times\bH_0=-i\omega\varepsilon_s\bE_0,  \quad \nabla\times\bE_0=i\omega\mu_0\bH_0,} & \mbox{in} \ \ \Sigma_s,\medskip\\
\displaystyle{ \nabla\times\bH_0=-i\omega\varepsilon_0\bE_0,  \quad \nabla\times\bE_0=i\omega\mu_0\bH_0,} & \mbox{in} \ \ \RR^3\setminus\overline{\Sigma},\medskip\\
 \displaystyle{\lim_{\|\Bx\|\rightarrow\infty} \|\Bx\| \big( \sqrt{\mu_0}\bH_0 \times\hat{\Bx}- \sqrt{\varepsilon_0}\bE_0\big)=0,}
 \end{array}
 \right .
\eeq
where $[{\nu} \times{\bE_0}]$ and $[{\nu} \times{\bH_0}]$ denote the jumps of ${\nu} \times{\bE_0}$
and ${\nu} \times{\bH_0}$ along $\p \Sigma_s$, namely,
 \begin{equation*}
 [\nu\times{\bE_0}]=(\nu\times \bE_0)\bigr|_+ -(\nu \times \bE_0)\bigr|_-,\quad [\nu\times{\bH_0}]=(\nu\times \bH_0)\bigr|_+ -(\nu\times \bH_0)\bigr|_-.
  \end{equation*}
Using the potential theory, the solution $(\bE_0,\bH_0)$ to \eqnref{eq:pss02}, outside $\Sigma_c$, can be represented by
\beq\label{eq:repre01}
\bE_0=\left\{\begin{split}
&\mu_0\nabla\times\Acal_{\Sigma_c}^{k_0}[\Phi_C]+\mu_0\nabla\times\Acal_{\Sigma}^{k_0}[\Phi_S]+\nabla\times\nabla \times\Acal_{\Sigma}^{k_0}[\Psi_S] \quad \mbox{in} \ \RR^3\setminus\overline{\Sigma}, \\
&\mu_0\nabla\times\Acal_{\Sigma_c}^{k_s}[\Phi_C]+\mu_0\nabla\times\Acal_{\Sigma}^{k_s}[\Phi_S]+\nabla\times\nabla \times\Acal_{\Sigma}^{k_s}[\Psi_S] \quad \mbox{in} \ \Sigma_s,
\end{split}
\right.
\eeq
and
\beq\label{eq:repre02}
\bH_0=\left\{\begin{split}
&-i\omega^{-1}\Big(\nabla\times\nabla\times\Acal_{\Sigma_c}^{k_0}[\Phi_C]+\nabla\times\nabla\times\Acal_{\Sigma}^{k_0}[\Phi_S]+\\
& \hspace*{4cm}\omega^2\varepsilon_0\nabla\times\Acal_{\Sigma}^{k_0}[\Psi_S]\Big) \ \mbox{in} \ \ \RR^3\setminus\overline{\Sigma}, \\
&-i\omega^{-1}\Big(\nabla\times\nabla\times\Acal_{\Sigma_c}^{k_s}[\Phi_C]+\nabla\times\nabla\times\Acal_{\Sigma}^{k_s}[\Phi_S]+\\
&\hspace*{4cm}\omega^2\varepsilon_s\nabla \times\Acal_{\Sigma}^{k_s}[\Psi_S]\Big)\ \mbox{in} \ \ \Sigma_s,
\end{split}
\right.
\eeq
where by the transmission conditions across $\p\Sigma$, $(\Phi_S,\Psi_S, \Phi_C)\in \mathrm{TH}(\mbox{div},\p \Sigma)\times \mathrm{TH}(\mbox{div},\p \Sigma)\times \mathrm{TH}(\mbox{div},\p \Sigma_c)$ satisfy
\beq\label{eq:Phiform01}
\left\{
\begin{split}
&\mu_0\Big(-I+\Mcal_{\Sigma}^{k_0}-\Mcal_{\Sigma}^{k_s}\Big)[\Phi_S]+\Big(\mathcal{L}_{\Sigma}^{k_0}-\mathcal{L}_{\Sigma}^{k_s}\Big)[\Psi_S]\\
=& \mu_0\big(\Mcal_{\Sigma,\Sigma_c}^{k_s}-\Mcal_{\Sigma,\Sigma_c}^{k_0}\big)[\Phi_C]\ \mbox{on} \ \ \p \Sigma, \\
&\omega^2\Big(-\frac{\varepsilon_s+\varepsilon_0}{2}I+\varepsilon_0\Mcal_{\Sigma}^{k_0}-\varepsilon_s\Mcal_{\Sigma}^{k_s}\Big)[\Psi_S]+
\Big(\mathcal{L}_{\Sigma}^{k_0}-\mathcal{L}_{\Sigma}^{k_s}\Big)[\Phi_S]\\
=& \big(\mathcal{L}_{\Sigma,\Sigma_c}^{k_s}-\mathcal{L}_{\Sigma,\Sigma_c}^{k_0}\big)[\Phi_C] \ \mbox{on} \ \ \p \Sigma,\\
&\mu_0\Big(-\frac{I}{2}+\Mcal_{\Sigma_c}^{k_s}\Big)[\Phi_C]+\mu_0\Mcal_{\Sigma_c,\Sigma}^{k_s}[\Phi_S]+\mathcal{L}_{\Sigma_c,\Sigma}^{k_s}[\Psi_S]
=\nu\times\bE_0\ \mbox{on} \ \ \p \Sigma_c,
\end{split}
\right.
\eeq
with $\Mcal_{\Sigma_c,\Sigma}^{k}$, $\mathcal{L}_{\Sigma_c,\Sigma}^{k}$, $\Mcal_{\Sigma,\Sigma_c}^{k}$ and $\mathcal{L}_{\Sigma,\Sigma_c}^{k}$ ($k=k_0, k_s$) defined by
\begin{align*}
\Mcal_{\Sigma,\Sigma_c}^{k}:= \nu\times\nabla\times\Acal_{\Sigma_c}^{k}\Big|_{\p \Sigma}, \quad& \mathcal{L}_{\Sigma,\Sigma_c}^{k}:=\nu\times\nabla\times\nabla\times\Acal_{\Sigma_c}^{k}\Big|_{\p \Sigma}, \\
\Mcal_{\Sigma_c,\Sigma}^{k}:= \nu\times\nabla\times\Acal_{\Sigma}^{k}\Big|_{\p \Sigma_c}, \quad&
\mathcal{L}_{\Sigma_c,\Sigma}^{k}:=\nu\times\nabla\times\nabla\times\Acal_{\Sigma}^{k}\Big|_{\p \Sigma_c}.
\end{align*}
By direct asymptotic analysis one can obtain that (see also \cite{ADM14,DHU:17})
\beq\label{eq:asymtmp01}
\Mcal_{\Sigma}^{k}=\Mcal_{\Sigma}^{0}+\Ocal(\omega^2), \quad \nabla\times\Acal_{\Sigma}^{k}=\nabla\times\Acal_{\Sigma}^{0}+\Ocal(\omega^2).
\eeq
One can also verify that
\beq\label{eq:asymtmp02}
\mathcal{L}_{\Sigma,\Sigma_c}^{k}=\nu\times D^2\Acal_{\Sigma_c}^{0}\Big|_{\p \Sigma}+k^2\nu\times(\Acal_{\Sigma_c}^{0}+D^2\mathcal{B}_{\Sigma_c})\Big|_{\p \Sigma}+\Ocal(\omega^3)
\eeq
and similar results hold for $\mathcal{L}_{\Sigma_c,\Sigma}^{k}$, $\Mcal_{\Sigma,\Sigma_c}^{k}$ and $\Mcal_{\Sigma,\Sigma_c}^{k}$. Here $D^2$ denotes $\nabla\nabla\cdot$ and $\mathcal{B}_{\Sigma_c}:\rm{TH}({\rm div},\p \Sigma_c)\rightarrow H^2(\Sigma \setminus\overline{\Sigma_c})^3$ is defined by
\beq\label{eq:defbdop01}
\mathcal{B}_{\Sigma_c}[\Phi](\Bx):=\frac{1}{4\pi}\int_{\p \Sigma_c}\|\Bx-\By\|\Phi(\By)ds_\By, \quad \Phi\in \rm{TH}({\rm div},\p \Sigma_c).
\eeq
We next show that \eqnref{eq:Phiform01} is uniquely solvable, and to that end we first prove a useful lemma.

\begin{lem}\label{le:pri01}
There holds the following asymptotic expansion
\beq\label{eq:lepri01}
\mathcal{L}_{\Sigma_c,\Sigma}^{k_s}[\Psi_S]=\nu\times\nabla\Scal_{\Sigma}^0[\nabla_{\p B}\cdot\Psi_S]\Big|_{\p \Sigma_c}+\Ocal(\omega^2\|\Psi_S\|_{\rm{TH}({\rm div},\p \Sigma)}).
\eeq
\end{lem}
\begin{proof}
By using \eqnref{eq:asymtmp02} and integration by parts, one has by straightforward asymptotic analysis that
\beq\label{eq:thpf04}
\begin{split}
\mathcal{L}_{\Sigma_c,\Sigma}^{k_s}[\Psi_S]=&\nu\times D^2\Acal_{\Sigma}^{0}[\Psi_S]\Big|_{\p \Sigma_c}+\Ocal(\omega^2\|\Psi_S\|_{\rm{TH}({\rm div},\p \Sigma)})\\
=&\nu\times\nabla\nabla\cdot\Acal_{\Sigma}^{0}[\Psi_S]\Big|_{\p \Sigma_c}+\Ocal(\omega^2\|\Psi_S\|_{\rm{TH}({\rm div},\p \Sigma)})\\
=&\nu\times\nabla\int_{\p \Sigma}\nabla\Gamma_0(\cdot-\By)\cdot\Psi_S(\By)ds_\By+\Ocal(\omega^2\|\Psi_S\|_{\rm{TH}({\rm div},\p \Sigma)}) \\
=&-\nu\times\nabla\int_{\p\Sigma}\nabla_\By\Gamma_0(\cdot-\By)\nabla_{\p B}\cdot\Psi_S(\By)ds_\By+\Ocal(\omega^2\|\Psi_S\|_{\rm{TH}({\rm div},\p \Sigma)})\\
=&\nu\times\nabla\int_{\p\Sigma}\Gamma_0(\cdot-\By)\nabla_{\p B}\cdot\Psi_S(\By)ds_\By+\Ocal(\omega^2\|\Psi_S\|_{\rm{TH}({\rm div},\p \Sigma)}).
\end{split}
\eeq
The proof is complete.
\end{proof}
\begin{lem}\label{le:uniqueinteg01}
$(\Phi_S,\Psi_S, \Phi_C)\in \mathrm{TH}({\rm div},\p \Sigma)\times \mathrm{TH}({\rm div},\p \Sigma)\times \mathrm{TH}({\rm div},\p \Sigma_c)$ is uniquely solvable in \eqnref{eq:Phiform01} for all sufficiently small $\omega\in\mathbb{R}_+$.
\end{lem}
\begin{proof}
Denote by $\nu$ the exterior unit normal vector on $\p \Sigma_c$ and $\tilde\nu$ the exterior unit normal vector on $\p \Sigma$.
First, we recall that $\Mcal_{D}^{k}$, $\mathcal{L}_{D}^{k}$ are bounded on $\rm{TH}({\rm div},\p D)$, for $D=\Sigma,\Sigma_c$ and $k=k_0,k_s$ (see, e.g. \cite{ADM14,T}). From \eqnref{eq:asymtmp01}, \eqnref{eq:asymtmp02} and the first equation in \eqnref{eq:Phiform01}, one can find that
\beq\label{eq:leunique01}
\|\Phi_S\|_{\rm{TH}({\rm div},\p \Sigma)}=\Ocal\left(\omega^2(\|\Psi_S\|_{\rm{TH}({\rm div},\p \Sigma)}+\|\Phi_C\|_{\rm{TH}({\rm div},\p \Sigma_c)})\right).
\eeq
By substituting \eqnref{eq:leunique01} into the second equation of \eqnref{eq:Phiform01} and using asymptotic analysis, one can obtain that
\beq\label{eq:leunique02}
\begin{split}
& \Big(-\frac{\varepsilon_s+\varepsilon_0}{2}I+(\varepsilon_0-\varepsilon_s)\Mcal_{\Sigma}^{0}+\Ocal(\omega^2)\Big)[\Psi_S]\\
 =& \mu_0(\varepsilon_s-\varepsilon_0)\Big(\tilde\nu\times(\Acal_{\Sigma_c}^{0}+D^2\mathcal{B}_{\Sigma_c})+\omega\tilde{\mathcal{B}}_{\Sigma_c}\Big)[\Phi_c],
\end{split}
\eeq
where $\tilde{\mathcal{B}}_{\Sigma_c}$ is a bounded operator from $\rm{TH}({\rm div},\p \Sigma_c)$ to $\rm{TH}({\rm div},\p \Sigma)$. By taking the surface divergence on $\p \Sigma$ of both sides of \eqnref{eq:leunique02} and using the formula $\nabla_{\p \Sigma}\cdot \Mcal_{\Sigma}^{k_0}=-(\Kcal_{\Sigma}^{k_0})^*\nabla_{\p \Sigma}\cdot$ (see \cite{ADM14,DHH:17}), one can further obtain that
\beq\label{eq:leunique03}
\Big(\frac{\varepsilon_s+\varepsilon_0}{2}I+(\varepsilon_0-\varepsilon_s)(\Kcal_{\Sigma}^{0})^*+\Ocal(\omega^2)\Big)[\nabla_{\p \Sigma}\cdot\Psi_S]
=\mu_0(\varepsilon_s-\varepsilon_0)\Big(\tilde\nu\cdot(\nabla\times\Acal_{\Sigma_c}^{0})-\omega\nabla_{\p \Sigma}\cdot\tilde{\mathcal{B}}_{\Sigma_c}\Big)[\Phi_c].
\eeq
Note that $(\varepsilon_s+\varepsilon_0)/2I+(\varepsilon_0-\varepsilon_s)(\Kcal_{\Sigma}^{0})^*$ is invertible on $L^2(\p \Sigma)$, one thus has
\beq\label{eq:leunique04}
\nabla_{\p \Sigma}\cdot\Psi_S=\mu_0\Big(\Gl_{\varepsilon} I-(\Kcal_{\Sigma}^{0})^*\Big)^{-1}\tilde\nu\cdot(\nabla\times\Acal_{\Sigma_c}^{0})[\Phi_c]+\Ocal(\omega\|\Phi_C\|_{\rm{TH}({\rm div},\p \Sigma_c)}),
\eeq
where $\Gl_{\varepsilon}$ is defined by
$$
\Gl_{\varepsilon}:=\frac{\varepsilon_s+\varepsilon_0}{2(\varepsilon_s-\varepsilon_0)}.
$$
Finally, by using \eqnref{eq:lepri01} and substituting \eqnref{eq:leunique01} and \eqnref{eq:leunique04} into the last equation of \eqnref{eq:Phiform01}, along with the help of \eqnref{eq:leunique02}, one can show by direct asymptotic analysis that
\beq\label{eq:leunique05}
\Big(-\frac{I}{2}+\Mcal_{\Sigma_c}^{0}+\nu\times\nabla\Scal_{\Sigma}^0\left(\Gl_{\varepsilon} I-(\Kcal_{\Sigma}^{0})^*\right)^{-1}\tilde\nu\cdot(\nabla\times\Acal_{\Sigma_c}^{0})+\Ocal(\omega)\Big)[\Phi_C]
=\mu_0^{-1}\nu\times\bE_0.
\eeq
Next, we prove the unique solvability of \eqnref{eq:leunique05} when $\omega\in\mathbb{R}_+$ is sufficiently small, which is equivalent to proving the invertibility of the operator
$$-\frac{I}{2}+\Mcal_{\Sigma_c}^{0}+\nu\times\nabla\Scal_{\Sigma}^0\left(\Gl_{\varepsilon} I-(\Kcal_{\Sigma}^{0})^*\right)^{-1}\tilde\nu\cdot(\nabla\times\Acal_{\Sigma_c}^{0})$$
on $\rm{TH}({\rm div},\p \Sigma_c)$. Note that $\nu\times\nabla\Scal_{\Sigma}^0\left(\Gl_{\varepsilon} I-(\Kcal_{\Sigma}^{0})^*\right)^{-1}\tilde\nu\cdot(\nabla\times\Acal_{\Sigma_c}^{0})$ and $\Mcal_{\Sigma_c}^{0}$ are compact operators on
$\rm{TH}({\rm div},\p \Sigma_c)$. By using the Fredholm theory, it is sufficient to prove that the following homogeneous equation possesses only a trivial solution,
\beq\label{eq:leunique06}
\Big(-\frac{I}{2}+\Mcal_{\Sigma_c}^{0}+\nu\times\nabla\Scal_{\Sigma}^0\left(\Gl_{\varepsilon} I-(\Kcal_{\Sigma}^{0})^*\right)^{-1}\tilde\nu\cdot(\nabla\times\Acal_{\Sigma_c}^{0})\Big)[\Phi]=0.
\eeq
By taking the surface divergence of \eqnref{eq:leunique06} one then has
\beq\label{eq:leunique07}
\Big(-\frac{I}{2}-(\Kcal_{\Sigma_c}^{0})^*\Big)[\nabla_{\p \Sigma_c}\cdot\Phi]=0.
\eeq
By using the invertibility of $\frac{I}{2}+(\Kcal_{\Sigma_c}^{0})^*$ on $L^2(\p \Sigma_c)$, one thus has $\nabla_{\p \Sigma_c}\cdot\Phi=0$.
It can be verified that there exists only a trial solution to the following system (see Appendix \ref{app:01})
\beq\label{eq:leunique07}
\left\{
\begin{array}{ll}
\displaystyle{\nabla\times\bE=0, \quad \nabla\cdot\bE=0}, & \mbox{in} \,(\RR^3\setminus\overline{\Sigma})\cup \Sigma_s,\medskip\\
\displaystyle{\tilde\nu\times\bE|_+=\tilde\nu\times\bE|_-,}  & \mbox{on} \, \p \Sigma,\medskip \\
\displaystyle{\varepsilon_0\tilde\nu\cdot\bE|_+=\varepsilon_s\tilde\nu\cdot\bE|_-, }& \mbox{on} \, \p \Sigma,\medskip \\
\displaystyle{\nu\times\bE|_+=0, \quad \int_{\p \Sigma_c}\nu\cdot\bE|_+ =0,} &\mbox{on} \p \Sigma_c,\medskip \\
\displaystyle{\bE(\Bx)=\Ocal(\|\Bx\|^{-2}),  \quad \|\Bx\|\rightarrow \infty.}
\end{array}
\right.
\eeq
Furthermore, one can verify that
\beq\label{eq:leunique08}
\bE=\Big(\nabla\times\Acal_{\Sigma_c}^{0}+\nabla\Scal_{\Sigma}^0\left(\Gl_{\varepsilon} I-(\Kcal_{\Sigma}^{0})^*\right)^{-1}\tilde\nu\cdot(\nabla\times\Acal_{\Sigma_c}^{0})\Big)[\Phi]
\eeq
is also the solution to \eqnref{eq:leunique07}. Hence,
\beq\label{eq:leunique09}
\Big(\nabla\times\Acal_{\Sigma_c}^{0}+\nabla\Scal_{\Sigma}^0\left(\Gl_{\varepsilon} I-(\Kcal_{\Sigma}^{0})^*\right)^{-1}\tilde\nu\cdot(\nabla\times\Acal_{\Sigma_c}^{0})\Big)[\Phi]=0\ \ \mbox{in} \ \ \RR^3\setminus\overline{\Sigma_c}.
\eeq
Then one has $\tilde\nu\cdot\bE=0$ on $\p \Sigma$, which together with the jump formula \eqnref{eq:trace} further implies that
$$
\tilde\nu\cdot(\nabla\times\Acal_{\Sigma_c}^{0}[\Phi])=0 \quad \mbox{on}\ \ \p \Sigma.
$$
Therefore by \eqnref{eq:leunique08} one has
$$
\nabla\times\Acal_{\Sigma_c}^{0}[\Phi]=0 \quad \mbox{in}\ \ \RR^3\setminus\overline\Sigma_c.
$$
Finally, one has $\Phi=0$ by using
$$
\nu\times\nabla\times\Acal_{\Sigma_c}^{0}[\Phi]=\Big(-\frac{I}{2}+\Mcal_{\Sigma_c}^{0}\Big)[\Phi]=0\ \ \mbox{on} \ \ \p \Sigma_c,
$$
and hence the invertibility of $-I/2+\Mcal_{\Sigma_c}^{0}$ on $\rm{TH}({\rm div}, \p \Sigma_c)$.

We have proved the unique solvability of \eqnref{eq:leunique05}. By using \eqnref{eq:leunique01} and \eqnref{eq:leunique02} one can thus find a unique solution to \eqnref{eq:Phiform01}.  This completes the proof.
\end{proof}

By virtue of Lemma \ref{le:uniqueinteg01}, we can derive the following result

\begin{lem}\label{le:stelec01}
Let $(\bE_0,\bH_0)$ be the solution to \eqnref{eq:pss02}. Then for $\omega\in\mathbb{R}_+$ sufficiently small, one has
\beq\label{eq:le0101}
\bH_0=\nabla\Scal_{\Sigma_c}^{0}\Big(\frac{I}{2}+(\Kcal_{\Sigma_c}^{0})^*\Big)^{-1}[\nu\cdot\bH_0|_{\p \Sigma_c}^{+}]+\Ocal(\omega)\ \ \rm{in} \ \ \RR^3\setminus\overline{\Sigma_c}.
\eeq
\end{lem}

\begin{proof}
First, by using Lemma \ref{le:uniqueinteg01}, one has
\beq\label{eq:lethpf0001}
\Phi_C
=\left(\mu_0^{-1}\Big(-\frac{I}{2}+\Mcal_{\Sigma_c}^{0}+\nu\times\nabla\Scal_{\Sigma}^0\left(\Gl_{\varepsilon} I-(\Kcal_{\Sigma}^{0})^*\right)^{-1}\tilde\nu\cdot(\nabla\times\Acal_{\Sigma_c}^{0})\Big)^{-1}+\Ocal(\omega)\right)[\nu\times\bE_0].
\eeq
Note that \eqnref{eq:leunique02} and \eqnref{eq:leunique01} imply
\beq\label{eq:lethpf0002}
\|\Psi_S\|_{\rm{TH}({\rm div},\p \Sigma)}=\Ocal(\|\Phi_C\|_{\rm{TH}({\rm div},\p \Sigma_c)}), \quad \|\Phi_S\|_{\rm{TH}({\rm div},\p \Sigma)}=\Ocal(\omega^2\|\Phi_C\|_{\rm{TH}({\rm div},\p \Sigma_c)}).
\eeq
Hence by using \eqnref{eq:repre02} and straightforward asymptotic analysis, there holds
\beq\label{eq:thpf01}
\bH_0=-i\omega^{-1}\nabla\times\nabla\times\Acal_{\Sigma_c}^{0}[\Phi_C]+\Ocal(\omega)
=-i\omega^{-1}\nabla\Scal_{\Sigma_c}^{0}[\nabla_{\p \Sigma_c}\cdot\Phi_C]+\Ocal(\omega) \ \mbox{in} \ \RR^3\setminus\overline{\Sigma_c}.
\eeq
Moreover, by using \eqnref{eq:lethpf0001} one can show that
\beq\label{eq:thpf02}
\nabla_{\p \Sigma_c}\cdot\Phi_C=i\omega\Big(\frac{I}{2}+(\Kcal_{\Sigma_c}^{0})^*\Big)^{-1}[\nu\cdot\bH_0|_{\p \Sigma_c}^{+}]+\Ocal(\omega^2)\ \mbox{on} \ \p \Sigma_c.
\eeq
By substituting \eqnref{eq:thpf02} into \eqnref{eq:thpf01}, one thus has \eqnref{eq:le0101}.

The proof is complete.
\end{proof}
Lemma \ref{le:stelec01} shows that the leading-order term in the low-frequency asymptotic expansion of the magnetic field generated by the Earth's core is a gradient filed, namely it is conservative. We would like to point out that the leading-order term of the low-frequency asymptotic expansion of the electric field can also be exactly calculated by following a similar argument. However, since our main concern is to use the monitoring of the magnetic field for detecting the anomalies, we choose not to give the details on that aspect.

\section{Integral representation and asymptotics of $\bH$}

In this section, we consider the case that the Earth's magnetic field is perturbed by the anomalous magnetized objects, that is, $\Gs$, $\varepsilon$ and $\mu$ are replaced by \eqnref{eq:paradef02}, respectively. Henceforth, we denote by $\bE$ and $\bH$, respectively, the associated electric and magnetic fields. In the following, we define the wave numbers $\varsigma_l$, $l=1, 2, \ldots, l_0$, by $\varsigma_l^2:=\omega^2\mu_l\gamma_l$, $\gamma_l:=\Ge_l+i\Gs_l/\omega$, where the sign of $\varsigma_l$ is chosen such that $\Im{\varsigma_l}\geq 0$ (see \cite{CK}). For the sake of simplicity, we denote by $\mathbb{S}^2$ the unit sphere and define $\tilde{D}:=\Sigma_s\setminus\overline{\bigcup_{l'=1}^{l_0}D_{l'}}$. We also let $\nu_l$ be the exterior unit normal vector defined on $\p D_l$, $l=1, 2, \ldots, l_0$. By using the integral ansatz, one can have the following representation formula, whose proof is postponed to be given in Appendix B.
\begin{lem}\label{le:rep01}
Let $(\bE,\bH)$ be the solution to \eqnref{eq:paradef02} and \eqnref{eq:pss01}. Then there hold the following results,
\beq\label{eq:solrepre01}
\bE=\left\{
\begin{split}
&\hat\bE_0+\nabla\times(\mu_0\Acal_\Sigma^{k_0}[\Phi_0]+\nabla\times\Acal_\Sigma^{k_0}[\Psi_0])\\
&+\nabla\times\sum_{l'=1}^{l_0}(\mu_0\Acal_{D_{l'}}^{k_0}[\Phi_{l'}]+\nabla\times\Acal_{D_{l'}}^{k_0}[\Psi_{l'}]) \quad \mbox{in} \ \ \RR^3\setminus\overline\Sigma,\\
&\hat\bE_0+\nabla\times(\mu_0\Acal_\Sigma^{k_s}[\Phi_0]+\nabla\times\Acal_\Sigma^{k_s}[\Psi_0])\\
&+\nabla\times\sum_{l'=1}^{l_0}(\mu_0\Acal_{D_{l'}}^{k_s}[\Phi_{l'}]+\nabla\times\Acal_{D_{l'}}^{k_s}[\Psi_{l'}]) \quad \mbox{in} \ \ \tilde{D},\\
& \nabla\times(\mu_l\Acal_\Sigma^{\varsigma_l}[\Phi_0]+\nabla\times\Acal_\Sigma^{\varsigma_l}[\Psi_0])\\
&+\nabla\times\sum_{l'=1}^{l_0}(\mu_l\Acal_{D_{l'}}^{\varsigma_l}[\Phi_{l'}]+\nabla\times\Acal_{D_{l'}}^{\varsigma_l}[\Psi_{l'}]) \quad \mbox{in} \ \ D_l,
\end{split}
\right.
\eeq
and
\beq\label{eq:solrepre02}
\bH=\left\{
\begin{split}
\hat\bH_0-i\omega^{-1}\nabla\times\Big((\omega^2\varepsilon_0\Acal_{\Sigma}^{k_0}[\Psi_0]+\nabla\times\Acal_{\Sigma}^{k_0}[\Phi_0])&\\
+\sum_{l'=1}^{l_0}(\omega^2\varepsilon_0\Acal_{D_{l'}}^{k_0}[\Psi_{l'}]+\nabla\times\Acal_{D_{l'}}^{k_0}[\Phi_{l'}])\Big) &\quad \mbox{in} \ \ \RR^3\setminus\overline\Sigma, \\
\hat\bH_0-i\omega^{-1}\nabla\times\Big((\omega^2\varepsilon_s\Acal_{\Sigma}^{k_s}[\Psi_0]+\nabla\times\Acal_{\Sigma}^{k_s}[\Phi_0])&\\
+\sum_{l'=1}^{l_0}(\omega^2\varepsilon_s\Acal_{D_{l'}}^{k_s}[\Psi_{l'}]+\nabla\times\Acal_{D_{l'}}^{k_s}[\Phi_{l'}])\Big) &\quad \mbox{in} \ \ \tilde{D}, \\
-i\omega^{-1}\nabla\times\Big((\omega^2\gamma_l\Acal_{\Sigma}^{\varsigma_l}[\Psi_0]+\nabla\times\Acal_{\Sigma}^{\varsigma_l}[\Phi_0])&\\
+\sum_{l'=1}^{l_0}(\omega^2\gamma_{l}\Acal_{D_{l'}}^{\varsigma_l}[\Psi_{l'}]+\nabla\times\Acal_{D_{l'}}^{\varsigma_l}[\Phi_{l'}])\Big) &\quad \mbox{in} \ \ D_l,
\end{split}
\right.
\eeq
where $(\Phi_0, \Psi_0)\in \rm{TH}({\rm div},\p \Sigma)\times \rm{TH}({\rm div},\p \Sigma)$ and $(\Phi_l, \Psi_l)\in \mathrm{TH}({\rm div}, \p D_l)\times\mathrm{TH}({\rm div}, \p D_l)$, $l=1, 2, \ldots, l_0$ satisfy \eqnref{phi_psi}. The fields $(\hat\bE_0, \hat\bH_0)$ satisfy \eqnref{eq:paradef01} and \eqnref{eq:pss01} in $(\RR^3\setminus\overline \Sigma)\bigcup\tilde{D}$ with $\hat\bH_0$ given in \eqref{eq:leH1repre01} in the following, and they depend on the background fields $(\bE_0, \bH_0)$ and the boundary condition on $\p \Sigma_c$.
\end{lem}

Based on Lemma~\ref{le:rep01}, we next derive two critical asymptotic expansions of the electromagnetic fields $\bE$ and $\bH$. The first one is the low-frequency asymptotics of the aforementioned fields, and the leading-order terms are referred to as the steady fields. The second one is the asymptotic expansion of the steady fields in terms of the size of the anomalies.

\subsection{First level approximation}

In this section, we derive the steady parts of the perturbed magnetic field in \eqref{eq:solrepre02}. We have the following asymptotic expansion results
\begin{thm}\label{th:0101}
Let $(\bE,\bH)$ be the solution to \eqnref{eq:paradef02} and \eqnref{eq:pss01}.
Then for $\omega\in\mathbb{R}_+$ sufficiently small, there hold the following asymptotic expansions,
\beq\label{eq:th0101}
\bH=\left\{
\begin{split}
\hat\bH_0&-\varepsilon_0\nabla\times\Acal_\Sigma^0[\Xi]+\nabla\Scal_{\Sigma}^0[\Theta]\\
&+\sum_{l'=1}^{l_0}\Big(\varepsilon_0\nabla\times\Acal_{D_{l'}}^{0}[\Psi_{l'}^{(0)}]-\mu_0\nabla\Scal_{D_{l'}}^{0}[\Pi_{l'}]\Big)+\Ocal(\omega) \quad \mbox{in} \ \ \RR^3\setminus\overline{\Sigma}, \\
\hat\bH_0&-\varepsilon_s\nabla\times\Acal_\Sigma^0[\Xi]+\nabla\Scal_{\Sigma}^0[\Theta]\\
&+\sum_{l'=1}^{l_0}\Big(\varepsilon_s\nabla\times\Acal_{D_{l'}}^{0}[\Psi_{l'}^{(0)}]-\mu_0\nabla\Scal_{D_{l'}}^{0}[\Pi_{l'}]\Big)+\Ocal(\omega) \quad \mbox{in} \ \ \tilde{D}, \\
&-\gamma_l\nabla\times\Acal_\Sigma^0[\Xi]-\nabla\Scal_{\Sigma}^0[\Theta]+\sum_{l'=1}^{l_0}\Big(\gamma_l\nabla\times\Acal_{D_{l'}}^{0}[\Psi_{l'}^{(0)}]-\mu_0\nabla\Scal_{D_{l'}}^{0}[\Pi_{l'}]\Big)+\Ocal(\omega) \ \mbox{in} \ D_l,
\end{split}
\right.
\eeq
where $\Xi, \Theta\in {\rm TH}({\rm div}, \p \Sigma)$ satisfy
\beq\label{eq:th010101}
\begin{split}
\Xi&=\sum_{l'=1}^{l_0}\Big(\Gl_{\varepsilon} I +\Mcal_{\Sigma}^0\Big)^{-1}\Mcal_{\Sigma, D_{l'}}^0[\Psi_{l'}^{(0)}], \\ \Theta&=(\varepsilon_s-\varepsilon_0)\sum_{l'=1}^{l_0}\nu\cdot\Big(\nabla\times\Acal_{D_{l'}}^0[\Psi_{l'}^{(0)}]\Big)
-(\varepsilon_s-\varepsilon_0)\nu\cdot\Big(\nabla\times\Acal_{\Sigma}^0[\Xi]\Big),
\end{split}
\eeq
and $\Psi_l^{(0)}\in {\rm TH}({\rm div}, \p D_l)$, $l=1, 2,\ldots, l_0$ are defined in \eqnref{eq:app17}.
$\Pi_l\in L^2(\p D_l)$, $l=1, 2,\ldots, l_0$ are defined by
\beq\label{eq:th0103}
\begin{split}
\Pi_l=&\Big((\mathbb{J}_D^{\mu})^{-1}\big[(\frac{\nu_1\cdot\hat\bH_0}{\mu_1-\mu_0}, \frac{\nu_2\cdot\hat\bH_0}{\mu_2-\mu_0}, \ldots, \frac{\nu_{l_0}\cdot\hat\bH_0}{\mu_{l_0}-\mu_0})^T\big]\Big)_l \\
&-\Big((\mathbb{J}_D^{\mu})^{-1}\big[(\frac{\omega\gamma_1\mu_1\nu_1\cdot\mathbf{C}}{\mu_1-\mu_0}, \frac{\omega\gamma_2\mu_2\mathbf{C}}{\mu_2-\mu_0}, \ldots, \frac{\omega\gamma_{l_1}\mu_{l_0}\nu_{l_0}\cdot\mathbf{C}}{\mu_{l_0}-\mu_0})^T\big]\Big)_l
\end{split}
\eeq
where $\mathbb{J}_D^{\mu}$ is defined in \eqnref{eq:thpf010303} and $\mathbf{C}$ is defined by
\beq\label{eq:th01add01}
\mathbf{C}:=\nabla\times\Acal_\Sigma^0\Big(\Gl_{\varepsilon} I +\Mcal_{\Sigma}^0\Big)^{-1}\sum_{l'=1}^{l_0}\Mcal_{\Sigma, D_{l'}}^0[\Psi_{l'}^{(0)}]
-\nabla\times\sum_{l'=1}^{l_0}\Acal_{D_{l'}}^0[\Psi_{l'}^{(0)}].
\eeq
The parameters $\Gl_{\mu_l}$ and $\Gl_{\gamma_l}$ are defined by
\beq\label{eq:defmulam}
\lambda_{\mu_l}:=\frac{\mu_l+\mu_0}{2(\mu_l-\mu_0)}, \quad \lambda_{\gamma_l}:=\frac{\gamma_l+\varepsilon_s}{2(\gamma_l-\varepsilon_s)}, \quad l=1, 2, \ldots, l_0.
\eeq
\end{thm}
\begin{proof}
By using direct asymptotic expansion with respect to $\omega$ in \eqnref{eq:solrepre02} and combing \eqnref{eq:app04}, \eqnref{eq:app05},
\eqnref{eq:app08} and \eqnref{eq:app16} one obtains that
\beq\label{eq:th0101001}
\bH=\left\{
\begin{split}
\hat\bH_0&-i(\varepsilon_0\nabla\times\Acal_\Sigma^0[\omega\Psi_0]+D^2\Acal_{\Sigma}^0[\omega^{-1}\Phi_0])\\
&-i\sum_{l'=1}^{l_0}\Big(\varepsilon_0\nabla\times\Acal_{D_{l'}}^{0}[\omega\Psi_{l'}]
+D^2\Acal_{D_{l'}}^{0}[\omega^{-1}\Phi_{l'}]\Big)+\Ocal(\omega)\ \mbox{in} \ \RR^3\setminus\overline{\Sigma}, \\
\hat\bH_0&-i(\varepsilon_s\nabla\times\Acal_\Sigma^0[\omega\Psi_0]+D^2\Acal_{\Sigma}^0[\omega^{-1}\Phi_0])\\
&-i\sum_{l'=1}^{l_0}\Big(\varepsilon_s\nabla\times\Acal_{D_{l'}}^{0}[\omega\Psi_{l'}]
+D^2\Acal_{D_{l'}}^{0}[\omega^{-1}\Phi_{l'}]\Big)+\Ocal(\omega)\ \mbox{in} \ \tilde{D}, &\\
&-i(\gamma_l\nabla\times\Acal_\Sigma^0[\omega\Psi_0]+D^2\Acal_{\Sigma}^0[\omega^{-1}\Phi_0])\\
&-i\sum_{l'=1}^{l_0}\Big(\gamma_l\nabla\times\Acal_{D_{l'}}^{0}[\omega\Psi_{l'}]
+D^2\Acal_{D_{l'}}^{0}[\omega^{-1}\Phi_{l'}]\Big)+\Ocal(\omega)\ \mbox{in} \ D_l.&
\end{split}
\right.
\eeq
Combining the first equation in \eqnref{phi_psi}, \eqnref{eq:app04}, \eqnref{eq:app05}, \eqnref{eq:app08} and using integration by parts one has
\beq\label{eq:th0101002}
\omega\Psi_0=-\Big(\Gl_{\varepsilon} I +\Mcal_{\Sigma}^0\Big)^{-1}\sum_{l'=1}^{l_0}\Mcal_{\Sigma, D_{l'}}^0[\omega\Psi_{l'}]
+\Ocal(\omega),
\eeq
and
\beq\label{eq:th0101002}
D^2\Acal_{\Sigma}^0[\omega^{-1}\Phi_0]=(\varepsilon_s-\varepsilon_0)\nabla\Scal_{\Sigma}^0
\left[\nu\cdot\Big(\nabla\times\sum_{l'=1}^{l_0}\Acal_{D_{l'}}^0[\omega\Psi_{l'}]\Big)+\nu\cdot\Big(\nabla\times\Acal_\Sigma^0[\omega\Psi_0]\Big)\right]
+\Ocal(\omega).
\eeq
Substituting \eqnref{eq:th0101002} into \eqnref{eq:th0101001} and using \eqnref{eq:app09} and \eqnref{eq:app16} one can show \eqnref{eq:th0101}.

The proof is complete.
\end{proof}

Now we present the explicit forms of the fields $\hat\bH_0$ and $\hat\bE_0$ in $(\RR^3\setminus\overline{\Sigma})\bigcup \tilde{D}$ that have been used in our earlier discussion. We remark that if there are no magnetized objects presented, namely $\mu_l=\mu_0$ and $\gamma_l=\varepsilon_s$, $l=1, 2, \ldots, l_0$, then \eqnref{eq:th0101} is degenerated to $\bH=\bH_0+\Ocal(\omega)$ in $\RR^3\setminus\overline\Sigma_c$. In this case, \eqnref{eq:app17} yields
\beq\label{eq:reextr01}
\Psi_l^{(0)}=\varepsilon_s^{-1}\nu_l\times\hat\bH_0\ \ \mbox{on} \ \ \p D_l, \quad l=1, 2, \ldots, l_0,
\eeq
and \eqnref{eq:th0103} yields
\beq\label{eq:reextr02}
\Pi_l=\mu_0^{-1}\nu_l\cdot\hat\bH_0+\Ocal(\omega)\ \ \mbox{on} \ \ \p D_l, \quad l=1, 2, \ldots, l_0.
\eeq
Substituting \eqnref{eq:reextr01} and \eqnref{eq:reextr02} into \eqnref{eq:th0101}, and using the assumption that $\bH=\bH_0+\Ocal(\omega)$, one obtains
\beq\label{eq:th01010101}
\begin{split}
\bH_0=&\hat\bH_0-\varepsilon_s^{-1}\sum_{l'=1}^{l_0}\Big(\varepsilon_0\nabla\times\Acal_\Sigma^0\Big(\Gl_{\varepsilon} I +\Mcal_{\Sigma}^0\Big)^{-1}\Mcal_{\Sigma, D_{l'}}^0[\nu_{l'}\times\hat\bH_0]\\
&-(\varepsilon_s-\varepsilon_0)\nabla\Scal_{\Sigma}^0\nu\cdot(\nabla\times\Acal_{D_{l'}}^0)[\nu_{l'}\times\hat\bH_0]\\
&+(\varepsilon_s-\varepsilon_0)\nabla\Scal_{\Sigma}^0\nu\cdot\nabla\times\Acal_\Sigma^0\Big(\Gl_{\varepsilon} I +\Mcal_{\Sigma}^0\Big)^{-1}\Mcal_{\Sigma, D_{l'}}^0[\nu_{l'}\times\hat\bH_0]\Big)\\
&+\sum_{l'=1}^{l_0}\Big(\varepsilon_0\varepsilon_s^{-1}\nabla\times\Acal_{D_{l'}}^{0}[\nu_{l'}\times\hat\bH_0]-\nabla\Scal_{D_{l'}}^{0}[\nu_{l'}\cdot\hat\bH_0]
\Big)+\Ocal(\omega)\ \mbox{in} \ \RR^3\setminus\overline{\Sigma}, \\
\end{split}
\eeq
and
\beq\label{eq:th01010102}
\begin{split}
\bH_0=&
\hat\bH_0-\varepsilon_s^{-1}\sum_{l'=1}^{l_0}\Big(\varepsilon_s\nabla\times\Acal_\Sigma^0\Big(\Gl_{\varepsilon} I +\Mcal_{\Sigma}^0\Big)^{-1}\Mcal_{\Sigma, D_{l'}}^0[\nu_{l'}\times\hat\bH_0]\\
&-(\varepsilon_s-\varepsilon_0)\nabla\Scal_{\Sigma}^0\nu\cdot(\nabla\times\Acal_{D_{l'}}^0)[\nu_{l'}\times\hat\bH_0]\\
&+(\varepsilon_s-\varepsilon_0)\nabla\Scal_{\Sigma}^0\nu\cdot\nabla\times\Acal_\Sigma^0\Big(\Gl_{\varepsilon} I +\Mcal_{\Sigma}^0\Big)^{-1}\Mcal_{\Sigma, D_{l'}}^0[\nu_{l'}\times\hat\bH_0]\Big)\\
&+\sum_{l'=1}^{l_0}\Big(\nabla\times\Acal_{D_{l'}}^{0}[\nu_{l'}\times\hat\bH_0]-\nabla\Scal_{D_{l'}}^{0}[\nu_{l'}\cdot\hat\bH_0]
\Big)+\Ocal(\omega)\ \mbox{in} \ \tilde{D}.
\end{split}
\eeq
We can further simplify \eqnref{eq:th01010101} and \eqnref{eq:th01010102} into some more compact form. To that end, we first derive the following lemma
\begin{lem}\label{le:eleresult010101}
There hold the following relations
\beq\label{eq:re010101}
\begin{split}
\nabla\times\Acal_{D_{l'}}^0[\nu_{l'}\times\hat\bH_0]=
\nabla\Scal_{D_l}^{0}[\nu_l\cdot\hat\bH_0]+\Ocal(\omega)\ \ & \qquad \mbox{in} \ \ (\RR^3\setminus\overline\Sigma)\bigcup \tilde{D},
\end{split}
\eeq
and
\beq\label{eq:re010102}
\begin{split}
&\nabla\times\Acal_\Sigma^0\Big(\Gl_{\varepsilon} I +\Mcal_{\Sigma}^0\Big)^{-1}\Mcal_{\Sigma, D_{l'}}^0[\nu_{l'}\times\hat\bH_0]\\
=&\left\{
\begin{array}{cc}
\nabla\Scal_\Sigma^0\Big(\Gl_{\varepsilon} I +(\Kcal_{\Sigma}^0)^*\Big)^{-1}\nu\cdot\nabla\Scal_{D_{l'}}^0[\nu_{l'}\cdot\hat\bH_0]+\Ocal(\omega) \ \mbox{in} \ \ \RR^3\setminus\overline{\Sigma},\\
\frac{\varepsilon_s-\varepsilon_0}{\varepsilon_s}\nabla\Scal_{D_{l'}}^0[\nu_{l'}\cdot\hat\bH_0]+\frac{\varepsilon_0}{\varepsilon_s}
\nabla\Scal_\Sigma^0\Big(\Gl_{\varepsilon} I +(\Kcal_{\Sigma}^0)^*\Big)^{-1}\nu\cdot\nabla\Scal_{D_{l'}}^0[\nu_{l'}\cdot\hat\bH_0]+\Ocal(\omega)\ \mbox{in} \ \ \tilde{D}.
\end{array}
\right.
\end{split}
\eeq
\end{lem}
\begin{proof}
Note that the lower order term of $\hat\bH_0$ is the gradient of a harmonic function (similar to \eqnref{eq:le0101}). The proof for \eqnref{eq:re010101} follows from a similar argument to that for the proof of Lemma \ref{le:eleresult01}. Using \eqnref{eq:re010101} and also a similar argument to that in the proof of Lemma \ref{le:eleresult01} one can then prove \eqnref{eq:re010102}.
\end{proof}

By Lemma \ref{le:eleresult010101}, \eqnref{eq:th01010102} and \eqref{eq:th01010101}, one can readily show that
\begin{lem}\label{le:eleresult010102}
$\hat\bH_0$ introduced in \eqnref{le:eleresult010101} satisfies
\beq\label{eq:leH1repre01}
\begin{split}
\bH_0&=\hat\bH_0+(\varepsilon_0-\varepsilon_s)\varepsilon_s^{-1}\sum_{l'=1}^{l_0}\nabla\Scal_{D_{l'}}^0[\nu_{l'}\cdot\hat\bH_0]+\Ocal(\omega) \ \ \mbox{in} \ \ (\RR^3\setminus\overline\Sigma)\bigcup \tilde{D}.
\end{split}
\eeq
\end{lem}
By using the jump formula on $\p D_l$, $l=1, 2, \ldots l_0$ and Lemma \ref{le:eleresult010102}, one can further obtain
\begin{lem}\label{le:eleresult010103}
$\hat\bH_0$ introduced in \eqnref{le:eleresult010101} can be written as
\beq\label{eq:leH1repre01new}
\begin{split}
\hat\bH_0&=\bH_0+\sum_{l'=1}^{l_0}\nabla\Scal_{D_{l'}}^0[\varphi_{l'}]+\Ocal(\omega) \ \ \mbox{in} \ \ (\RR^3\setminus\overline\Sigma)\bigcup \tilde{D},
\end{split}
\eeq
where $\varphi_{l}$, $l=1, 2, \ldots, l_0$ satisfy
\beq\label{eq:leeleresult010103}
\varphi_l=\Big((\mathbb{J}_D^{\varepsilon})^{-1}\big[(\nu_1\cdot\bH_0, \nu_2\cdot\bH_0, \ldots, \nu_{l_0}\cdot\bH_0)^T\big]\Big)_l,
\eeq
with the operator $\mathbb{J}_D^{\varepsilon}$ defined in \eqnref{eq:app10} with $\Gl_{\mu_l}$ replaced by $\Gl_\varepsilon$, $l=1, 2, \ldots, l_0$.
\end{lem}
\subsection{Second level approximation}\label{sec:02}

In the subsequent analysis, we shall make use of the steady part of the magnetic field in the representation formula \eqref{eq:th0101}, namely the leading-order term in the asymptotic low-frequency expansion. By Lemma \ref{le:stelec01}, we see that $\bH_0$ is a gradient filed in $\RR^3\setminus\overline{B}$. In what follows, we let $\bH^{(0)}$ be the leading-order term of $\bH$ in \eqnref{eq:th0101}, $\bH^{(0)}_0$ be the leading-term of $\bH_0$ and $\hat\bH_0^{(0)}$ be the leading-order term of $\hat\bH_0$ (and $\varphi_l^{(0)}$ is the leading term of $\varphi_l$, $l=1, 2, \ldots, l_0$). We would like to point out that if $\Gs_l$, $l=1,2,\ldots, l_0$ are not identically zero, then the leading-order term in \eqnref{eq:th0101} may still depend on $\omega$. In such a case, one needs to perform further asymptotic analysis in terms of the frequency and we shall discuss this point at the suitable place in what follows.

In this section, we consider further asymptotic expansion of the steady fields in terms of the size of the magnetized anomalies. Indeed, from a practical point of view, the size of the magnetized anomalies $(D_l; \varepsilon_l, \mu_l, \sigma_l)$, $l=1,2,\ldots,l_0$, introduced in \eqref{eq:paradef02}, is much smaller than the size of the Earth. Hence, we can assume that
\beq\label{eq:permeab02}
D_l=\delta\Omega+\bZ_l, \quad l=1, 2, \ldots, l_0,
\eeq
where $\Omega$ is a bounded Lipschitz domain in $\RR^3$ and $\Omega\subset\subset \Sigma$, and $\delta\in\mathbb{R}_+$ is sufficiently small.
Furthermore, we assume that $D_l$, $l=1, 2, \ldots, l_0$ are sparsely distributed and far away from each other and $\Bz_l$, $l=1, 2, \ldots, l_0$, are far away from $\p \Sigma$ such that $\Bx-\Bz_l\gg\delta$, for any $\Bx\in\p \Sigma$. With the above preparations, we are in a position to derive the asymptotic expansion of the steady geomagnetic field in terms of the size of the magnetized anomalies. We first have the following lemma
\begin{lem}\label{le:smallobj01}
Suppose $D_l$, $l=1, 2 ,\ldots, l_0$ are defined in \eqnref{eq:permeab02} with $\delta\in\mathbb{R}_+$ sufficiently small. Let $\mathbb{K}_D^*$,  and $\mathbb{M}_D$ be defined in \eqnref{eq:app11} and \eqnref{eq:thpf010302}, respectively. Then we have
\beq\label{eq:lesmobj0101}
\mathbb{K}_D^*=\mathbb{K}_\Omega^*+\Ocal(\delta^2), \quad \mathbb{M}_D=\mathbb{M}_{\Omega} +\Ocal(\delta^2),
\eeq
where $\mathbb{K}_\Omega^*$, $\mathbb{M}_{\Omega}$ are $l_0\times l_0$ and $3l_0\times 3l_0$ matrix-valued operators defined by
\beq\label{eq:lesmobj0102}
\mathbb{K}_\Omega^*:={\rm diag}((\Kcal_{\Omega}^{0})^*,(\Kcal_{\Omega}^{0})^*,\ldots,(\Kcal_{\Omega}^{0})^*), \quad \mathbb{M}_\Omega:={\rm diag}(\Mcal_{\Omega}^{0},\Mcal_{\Omega}^{0},\ldots,\Mcal_{\Omega}^{0}),
\eeq
respectively.
\end{lem}
\begin{proof}
We only prove the second assertion in \eqnref{eq:lesmobj0101}, and the first one can be proved in a similar manner. For any $\Bx, \By\in \p D_l$, we let $\tdx=\delta^{-1}(\Bx-\Bz_l)$, $\tdy=\delta^{-1}(\By-\Bz_l) \in \p  \Omega$, $l=1, 2,\ldots, l_0$. Define $\tilde{\Phi}(\tdy)= \Phi(\By)$.
By using change of variables, one can show that there holds
\beq\label{eq:lesmobj0104}
\begin{split}
\Mcal_{D_l}^0[\Phi](\Bx)=&\nu_\Bx  \times \nabla_\Bx \times \int_{\p D_l} \Gamma_{0}(\Bx-\By) \Phi(\By) d s_\By \\
=& -\delta^2\frac{1}{4\pi}\nu_\Bx\times \nabla_\Bx\times\int_{\p \Omega} \frac{1}{\|\Bx-\By\|}\tilde\Phi(\tdy) d s_{\tdy} \\
=& -\frac{1}{4\pi}\nu_{\tdx}\times \nabla_{\tdx}\times\int_{\p \Omega} \frac{1}{\|\tdx-\tdy\|}\tilde\Phi(\tdy) d s_{\tdy}=\Mcal_{\Omega}^0[\tilde\Phi](\Bx).
\end{split}
\eeq
On the other hand, letting $\Bx\in \p D_l$ and $\By\in \p D_m$ and $\tdx=\delta^{-1}(\Bx-\Bz_l)$, $\tdy=\delta^{-1}(\By-\Bz_m) \in \p  \Omega$, where $l, m\in \{1, 2, \ldots, l_0\}$ and $l\neq m$, then one can show that
\beq\label{eq:lesmobj0105}
\begin{split}
\Mcal_{D_l, D_m}^{0}[\Phi](\Bx)=&\nu_\Bx\times\nabla_\Bx\times\Acal_{D_m}^{0}[\Phi](\Bx)\\
=&\nu_\Bx  \times \nabla_\Bx \times \int_{\p D_m} \Gamma_{0}(\Bx-\By) \Phi(\By) d s_\By \\
=& -\delta^2\frac{1}{4\pi}\nu_\Bx\times \nabla_\Bx\times\int_{\p \Omega} \frac{1}{\|\Bx-\By\|}\tilde\Phi(\tdy) d s_{\tdy} \\
=& -\delta\frac{1}{4\pi}\nu_{\tdx}\times \nabla_{\tdx}\times\int_{\p \Omega} \frac{1}{\|\delta(\tdx-\tdy)+(\Bz_l-\Bz_m)\|}\tilde\Phi(\tdy) d s_{\tdy}=\Ocal(\delta^2).
\end{split}
\eeq
By substituting \eqnref{eq:lesmobj0104} and \eqnref{eq:lesmobj0105} back into \eqnref{eq:thpf010302}, one readily has \eqnref{eq:lesmobj0101}.

The proof is complete.
\end{proof}

\begin{lem}\label{le:eleresult01}
For any simply connected domain $D_l$ and the gradient filed $\bH_0^{(0)}$ in $\RR^3$, which is divergence free, there holds the following relation
\beq\label{eq:re0102}
\begin{split}
&\frac{1}{\gamma_l-\varepsilon_s}\nabla\times\Acal_{D_l}^{0}(\lambda_{\gamma_l}I+\Mcal_{D_l}^0)^{-1}[\nu_l\times\bH_0^{(0)}]\\
=&\left\{
\begin{array}{cc}
\displaystyle{\frac{1}{\gamma_l-\varepsilon_s}\nabla\Scal_{D_l}^{0}(\lambda_{\gamma_l}I+(\Kcal_{D_l}^0)^*)^{-1}[\nu_l\cdot\bH_0^{(0)}] }&\qquad\ \ \mbox{in} \ \ \RR^3\setminus\overline{D_l},\bigskip\\
\displaystyle{\frac{1}{\gamma_l}\bH_0^{(0)}
+\frac{\varepsilon_s}{\gamma_l(\gamma_l-\varepsilon_s)}\nabla\Scal_{D_l}^{0}(\lambda_{\gamma_l}I+(\Kcal_{D_l}^0)^*)^{-1}[\nu_l\cdot\bH_0^{(0)}]} &\ \ \mbox{in} \ \ D_l.
\end{array}
\right.
\end{split}
\eeq
\end{lem}
\begin{proof}
Note that $\bH_0^{(0)}$ is the gradient of a harmonic function. The proof of \eqnref{eq:re0102} follows from a similar argument to that in the proof of Lemma 5.5 in \cite{ADM14}.
\end{proof}
\begin{thm}\label{le:maggra01}
Suppose $D_l$, $l=1, 2 ,\ldots, l_0$ are defined in \eqnref{eq:permeab02} with $\delta\in\mathbb{R}_+$ sufficiently small. Let $(\bE,\bH)$ be the solution to \eqnref{eq:paradef02} and \eqnref{eq:pss02}. Then for $\Bx\in \RR^3\setminus\Sigma$, there holds the following asymptotic expansion result
\beq\label{eq:lemag0101}
\begin{split}
&\bH^{(0)}(\Bx)=
\bH_0^{(0)}(\Bx)-\delta^3\sum_{l=1}^{l_0}\nabla\big(\nabla\Gamma_0(\Bx-\Bz_l)^T\mathbf{P}_0\bH_0^{(0)}(\Bz_l)\big)\\
&-\delta^3\sum_{l=1}^{l_0}\Big(\varepsilon_0\nabla\big(\nabla\Gamma_0(\Bx-\Bz_l)^T\mathbf{D}_l\bH_0^{(0)}(\Bz_l)\big) -\mu_0\nabla\big(\nabla\Gamma_0(\Bx-\Bz_l)^T\mathbf{M}_l\bH_0^{(0)}(\Bz_l)\big)\Big)+\Ocal(\delta^4),
\end{split}
\eeq
where $\mathbf{P}_0$ is defined by
\beq\label{eq:leasymdefP0}
\mathbf{P}_0:=\int_{\p \Omega}\tilde\By(\lambda_{\varepsilon} I-(\Kcal_{\Omega}^{0})^*)^{-1}[\nu_l]ds_{\tilde\By}.
\eeq
The polarization tensors $\mathbf{D}_l$ and $\mathbf{M}_l$ are $3\times 3$ matrices defined by
\beq\label{eq:leasym02}
\mathbf{D}_l=\frac{1}{\gamma_l-\varepsilon_s}\frac{\varepsilon_s}{\varepsilon_s-\varepsilon_0}\int_{\p \Omega}\tilde\By(\lambda_{\gamma_l} I+(\Kcal_{\Omega}^{0})^*)^{-1}(\lambda_{\varepsilon} I-(\Kcal_{\Omega}^{0})^*)^{-1}[\nu_l]ds_{\tilde\By},
\eeq
and
\beq\label{eq:leasym03}
\mathbf{M}_l=\frac{1}{\mu_l-\mu_0}\frac{\varepsilon_s}{\varepsilon_s-\varepsilon_0}\int_{\p \Omega}\tilde\By(\lambda_{\mu_l} I-(\Kcal_{\Omega}^{0})^*)^{-1}(\lambda_{\varepsilon} I-(\Kcal_{\Omega}^{0})^*)^{-1}[\nu_l]ds_{\tilde\By},
\eeq
respectively, $l=1, 2, \ldots, l_0$. More specifically, let $\mathbf{D}_l=((\mathbf{D}_l)_{mn})$, $m,n=1, 2, 3$, $\tdy=(\tdy_1, \tdy_2, \tdy_3)^T$ and $\nu_l=(\nu_l^{(1)}, \nu_l^{(2)}, \nu_l^{(3)})^T$, we have
$$(\mathbf{D}_l)_{mn}=\frac{1}{\gamma_l-\varepsilon_s}\frac{\varepsilon_s}{\varepsilon_s-\varepsilon_0}\int_{\p \Omega}\tilde\By_m(\lambda_{\gamma_l} I+(\Kcal_{\Omega}^{0})^*)^{-1}(\lambda_{\varepsilon} I-(\Kcal_{\Omega}^{0})^*)^{-1}[\nu_l^{(n)}]ds_{\tilde\By},$$
where $e_{mn}=1$ for $m=n$ and $e_{mn}=0$ for $m\neq n$. $\mathbf{P}_0$ and $\mathbf{M}_l$ have similar forms.
\end{thm}
\begin{proof}
First,  we note that either $\omega\gamma_l$ or $1/(\gamma_l-\varepsilon_l)$ is of order $\omega$, $l=1, 2, \ldots l_0$, no matter $\Gs$ is zero or nonzero. One can immediately find that the second term in \eqnref{eq:th0103} is of order $\omega$.
By \eqnref{eq:th0101}, it then can be seen that the leading-order term $\bH^{(0)}$ has the following form
\beq\label{eq:leador0101}
\begin{split}
\bH^{(0)}=&\hat\bH_0^{(0)}-\varepsilon_0\sum_{l=1}^{l_0}\nabla\times\Acal_\Sigma^0\Big(\Gl_{\varepsilon} I +\Mcal_{\Sigma}^0\Big)^{-1}\nu\times\nabla\times\Acal_{D_{l}}^0[\Theta_{l}']\\
&+(\varepsilon_s-\varepsilon_0)\sum_{l=1}^{l_0}\nabla\Scal_{\Sigma}^0\Big[\nu\cdot\Big(\nabla\times\Acal_{D_{l}}^0[\Theta_{l}']\Big)\Big]\\
&-(\varepsilon_s-\varepsilon_0)\sum_{l=1}^{l_0}\nabla\Scal_{\Sigma}^0\nu\cdot\nabla\times\Acal_\Sigma^0\Big(\Gl_{\varepsilon} I +\Mcal_{\Sigma}^0\Big)^{-1}\nu\times\nabla\times\Acal_{D_{l}}^0[\Theta_{l}']\\
&+\sum_{l=1}^{l_0}\Big(\varepsilon_0\nabla\times\Acal_{D_{l}}^{0}[\Theta_{l}']-\mu_0\nabla\Scal_{D_{l}}^{0}[\Pi_{l}']\Big) \quad \mbox{in} \ \ \ \RR^3\setminus\Sigma,
\end{split}
\eeq
where $\Theta_l'$ and $\Pi_{l}'$ are defined by
$$\Theta_l'=\Big((\mathbb{L}_D^{\gamma})^{-1}\big[(\frac{(\nu_1\times\hat\bH_0^{(0)})^T}{\gamma_1-\varepsilon_s}, \frac{(\nu_2\times\hat\bH_0^{(0)})^T}{\gamma_2-\varepsilon_s}, \ldots, \frac{(\nu_{l_0}\times\hat\bH_0^{(0)})^T}{\gamma_{l_0}-\varepsilon_s})^T\big]\Big)\cdot(\mathbf{e}_l\otimes(1, 1, 1)^T),$$
and
$$
\Pi_l'=\Big((\mathbb{J}_D^{\mu})^{-1}\big[(\frac{\nu_1\cdot\hat\bH_0^{(0)}}{\mu_1-\mu_0}, \frac{\nu_2\cdot\hat\bH_0^{(0)}}{\mu_2-\mu_0}, \ldots, \frac{\nu_{l_0}\cdot\hat\bH_0^{(0)}}{\mu_{l_0}-\mu_0})^T\big]\Big)_l,
$$
respectively, $l=1, 2, \ldots, l_0$. It can be verified that
$$\nabla_{\p D_l}\cdot \Theta_l'=0, \quad l=1, 2, \ldots, l_0.$$
Thus $\nabla\times\Acal_{D_{l}}^0[\Theta_{l}']$ is a gradient field of harmonic function in $\RR^3\setminus\overline{D}_l$. By using
Lemma \ref{le:eleresult010101}, one can derive that
\beq\label{eq:leador0101new}
\begin{split}
\bH^{(0)}=&\hat\bH_0^{(0)}+\sum_{l=1}^{l_0}\Big(\varepsilon_0\nabla\times\Acal_{D_{l}}^{0}[\Theta_{l}']-\mu_0\nabla\Scal_{D_{l}}^{0}[\Pi_{l}']\Big) \quad \mbox{in} \ \ \ \RR^3\setminus\Sigma,
\end{split}
\eeq
As before, for $\By\in \p D_l$, we let $\tdy=\delta^{-1}(\By-\Bz_l) \in \p  \Omega$, and define $\tilde\Theta_l'(\tdy):=\Theta_l'(\By)$, $\tilde\Pi_l'(\tdy):=\Pi_l'(\By)$, $l=1, 2,\ldots, l_0$ and $\tilde{\hat\bH}_0^{(0)}(\tdy):=\hat\bH_0^{(0)}(\By)$. Then by Lemma \ref{le:smallobj01}, one has
\beq\label{eq:leasym0101}
\tilde\Theta_l'(\tdy)=\frac{1}{\gamma_l-\varepsilon_s}(\lambda_{\gamma_l} I+\Mcal_\Omega^0)^{-1}[\nu_l\times\tilde{\hat\bH}_0^{(0)}](\tdy)+\Ocal(\delta^2),
\eeq
and
\beq\label{eq:leasym0102}
\tilde\Pi_l'(\tdy)=\frac{1}{\mu_l-\mu_0}(\lambda_{\mu_l} I-(\Kcal_\Omega^0)^*)^{-1}[\nu_l\cdot\tilde{\hat\bH}_0^{(0)}](\tdy)+\Ocal(\delta^2),
\eeq
$l=1, 2, \ldots, l_0$.
Hence by using \eqnref{eq:re0102}, there holds
\beq\label{eq:leasym010101}
\begin{split}
\nabla\times\Acal_{D_{l}}^0[\Theta_{l}']=&\frac{1}{\gamma_l-\varepsilon_s}\nabla\times\Acal_{D_{l}}^0(\lambda_{\gamma_l} I+\Mcal_{D_l}^0)^{-1}[\nu_l\times\hat\bH_0^{(0)}](\tdy)+\Ocal(\delta^4)\\
=&\frac{1}{\gamma_l-\varepsilon_s}\nabla\Scal_{D_l}^0(\lambda_{\gamma_l} I+(\Kcal_{D_l}^0)^*)^{-1}[\nu_l\cdot\hat\bH_0^{(0)}](\tdy)+\Ocal(\delta^4)\\
:=&\nabla\Scal_{D_l}^0[Q_l]+\Ocal(\delta^4)\quad \mbox{in} \ \ \RR^3\setminus\Sigma.
\end{split}
\eeq
On the other hand, by the Taylor expansion, there holds
\beq\label{eq:leasym0103}
\bH_0^{(0)}(\By)=\bH_0^{(0)}(\Bz_l)+\delta\nabla\bH_0^{(0)}(\Bz_l)\tdy+\Ocal(\delta^2),
\eeq
and so by using \eqnref{eq:leH1repre01new} one has
\beq\label{eq:leasym0103new}
\nu\cdot\tilde{\hat\bH}_0^{(0)}(\tdy)=\nu_l\cdot\hat\bH_0^{(0)}(\By)=\nu_l\cdot\bH_0^{(0)}(\Bz_l)+\Big(\frac{I}{2}+(\Kcal_\Omega^0)^*\Big)[\tilde\varphi_l^{(0)}](\tdy)+\Ocal(\delta),
\eeq
where $\tilde\varphi_l^{(0)}(\tilde\By):=\varphi_l^{(0)}(\By)=\varphi_l(\By)+\Ocal(\omega)$ and by \eqnref{eq:leeleresult010103} and \eqnref{eq:leasym0103} one has
\beq\label{eq:leasym0104new}
\tilde\varphi_l^{(0)}(\tilde\By)=\Big(\Gl_\varepsilon I -(\Kcal_\Omega^0)^*\Big)^{-1}[\nu_l\cdot\bH_0^{(0)}]+\Ocal(\delta^2)=\Big(\Gl_\varepsilon I -(\Kcal_\Omega^0)^*\Big)^{-1}[\nu_l\cdot\bH_0^{(0)}(\Bz_l)]+\Ocal(\delta).
\eeq
For $\Bx-\Bz_l\gg\delta$ there also holds
\beq\label{eq:leasym0104}
\Gamma_0(\Bx-\By)=\Gamma_0(\Bx-\Bz_l)-\delta\nabla\Gamma_0(\Bx-\Bz_l)^T\tdy+\Ocal(\delta^2).
\eeq
Define $\tilde{Q}_l(\tdy):=Q_l(\By)$, where $Q_l$ is given in \eqnref{eq:leasym010101}.
By using change of variables and substituting \eqnref{eq:leasym0101}-\eqnref{eq:leasym0104} into \eqnref{eq:leador0101new} and using \eqnref{eq:leH1repre01new}, one thus has
\beq\label{eq:leasym0105}
\begin{split}
\bH^{(0)}(\Bx)=&\bH_0^{(0)}(\Bx)-\delta^3\sum_{l=1}^{l_0}\nabla^2\Gamma_0(\Bx-\Bz_l)\mathbf{P}_0\bH_0^{(0)}(\Bz_l)
-\delta^3\sum_{l=1}^{l_0}\varepsilon_0\int_{\p \Omega}\nabla^2\Gamma_0(\Bx-\Bz_l)^T\tdy \tilde Q_l \\
&+\delta^3\sum_{l=1}^{l_0}\Big(\mu_0\frac{1}{\mu_l-\mu_0}\int_{\p \Omega}\nabla^2\Gamma_0(\Bx-\Bz_l)\tdy(\lambda_{\mu_l} I-(\Kcal_\Omega^0)^*)^{-1}[\nu\cdot\tilde{\hat\bH}_0^{(0)}](\tdy)\Big)+\Ocal(\delta^4), \\
=&\bH_0^{(0)}(\Bx)-\delta^3\sum_{l=1}^{l_0}\nabla^2\Gamma_0(\Bx-\Bz_l)\mathbf{P}_0\bH_0^{(0)}(\Bz_l)
-\delta^3\sum_{l=1}^{l_0}\nabla^2\Gamma_0(\Bx-\Bz_l)\mathbf{D}_l\bH_0^{(0)}(\Bz_l)\\
&+\delta^3\sum_{l=1}^{l_0}\nabla^2\Gamma_0(\Bx-\Bz_l)\mathbf{M}_l\bH_0^{(0)}(\Bz_l)+\Ocal(\delta^4). \\
\end{split}
\eeq
The first equality of \eqnref{eq:leasym0105} is obtained by using the following fact
\beq\label{eq:leasym0106}
\int_{\p \Omega}\tilde Q_l=\Ocal(\delta^2), \quad \int_{\p \Omega}\tilde\Pi_l'=\Ocal(\delta^2).
\eeq
Indeed, in order to show \eqref{eq:leasym0106}, we set
$$
\phi_l(\tdy):=(\lambda_{\gamma_l} I+(\Kcal_\Omega^0)^*)^{-1}[\nu_l\cdot\tilde{\hat\bH}_0^{(0)}](\tdy),
$$
By using the jump formula \eqnref{eq:trace} and integration by parts, one can show that there holds
\[
\begin{split}
0=&\int_{\p \Omega}\nu_l\cdot\tilde{\hat\bH}_0^{(0)}=\int_{\p \Omega}(\lambda_{\gamma_l} I+(\Kcal_\Omega^0)^*)[\phi_l]\\
=&(\lambda_{\gamma_l}+1/2)\int_{\p \Omega}\phi_l-\int_{\p \Omega}(1/2I-(\Kcal_\Omega^0)^*)[\phi_l]\\
=&(\lambda_{\gamma_l}+1/2)\int_{\p \Omega}\phi_l-\int_{\p \Omega}\nu\cdot\nabla\Scal_\Omega^0[\phi_l]\Big|_-
=(\lambda_{\gamma_l}+1/2)\int_{\p \Omega}\phi_l,
\end{split}
\]
which readily proves the first assertion in \eqnref{eq:leasym0106}. The second assertion in \eqnref{eq:leasym0106} can be proven in a similar manner.

The proof is complete.
\end{proof}
For notational convenience, in the sequel, we introduce the matrix $\mathbf{P}_l$ by
\beq\label{eq:matM01}
\mathbf{P}_l:=\mu_0\mathbf{M}_l-\varepsilon_0\mathbf{D}_l-\mathbf{P}_0,
\eeq
where $\mathbf{M}_l$ and $\mathbf{D}_l$ are defined in \eqnref{eq:leasym02} and \eqnref{eq:leasym03}, respectively. We have the following axillary results
\begin{lem}\label{le:ax0000101}
If $\Gs_l\neq 0$, $l=1,2,\ldots, l_0$ and $\varepsilon_s=\varepsilon_0$, then $\mathbf{P}_l=\mu_0\mathbf{M}_l+\Ocal(\omega)$ is nonsingular.
\end{lem}
\begin{proof}
Since $\varepsilon_s=\varepsilon_0$, one immediately has $\mathbf{P}_0=0$ from \eqnref{eq:leasymdefP0}. Recall that $\gamma_l=\varepsilon_l+i\Gs_l\omega^{-1}$.
Since $\varepsilon_s=\varepsilon_0$ and $\Gs\neq 0$, it is straightforward to see from the definition of $\mathbf{D}_l$ in \eqnref{eq:leasym02} that
$$
\mathbf{D}_l=-i\omega\Gs_l^{-1}\int_{\p \Omega}\tilde\By\Big(\frac{I}{2}+(\Kcal_{\Omega}^{0})^*\Big)^{-1}[\nu]ds_{\tilde\By}+\Ocal(\omega^2).
$$
Then one can obtain that
\beq\label{eq:lenonsig01}
\begin{split}
\mathbf{P}_l=\mu_0\mathbf{M}_l+\Ocal(\omega)=&\frac{\mu_0}{\mu_l-\mu_0}\int_{\p \Omega}\tilde\By(\lambda_{\mu_l} I-(\Kcal_{\Omega}^{0})^*)^{-1}[\nu_l]ds_{\tilde\By}+\Ocal(\omega).
\end{split}
\eeq
It is known that the polarization tensor $\mathbf{M}_l$ in \eqnref{eq:lenonsig01} is a positive definite matrix (see, e.g.,\cite{ADKL14,HK07:book}).

The proof is complete.
\end{proof}

\begin{lem}\label{le:ax000}
Suppose $\Omega$ is a ball. Let $\mathbf{P}_l$ be defined in \eqnref{eq:matM01}. If there holds
\beq\label{eq:leax0101}
\mu_l\varepsilon_s^2+2(\mu_0-\mu_l)\varepsilon_s\gamma_l+2(\mu_l+2\mu_0)\varepsilon_0\gamma_l\neq \mu_0\varepsilon_s^2
\eeq
then $\mathbf{P}_l$ is nonsingular.
\end{lem}
\begin{proof}
Since $\Omega$ is a ball, one has the following result (see, e.g. \cite{HCKL13,FDL15})
\beq
(\Kcal_{\Omega}^{0})^*[\nu]=\frac{1}{6} \nu.
\eeq
Then one can calculate explicitly that
\beq\label{eq:reviseadd0101}
\begin{split}
\mathbf{P}_l=&\frac{\mu_0}{\mu_l-\mu_0}\frac{\varepsilon_s}{\varepsilon_s-\varepsilon_0}\int_{\p \Omega}\tilde\By(\lambda_{\mu_l} -1/6)^{-1}(\lambda_{\varepsilon}-1/6)^{-1}[\nu_l]ds_{\tilde\By}\\
&-\frac{\varepsilon_0}{\gamma_l-\varepsilon_s}\frac{\varepsilon_s}{\varepsilon_s-\varepsilon_0}\int_{\p \Omega}\tilde\By(\lambda_{\gamma_l} +1/6)^{-1}(\lambda_{\varepsilon}-1/6)^{-1}[\nu_l]ds_{\tilde\By}-\int_{\p \Omega}\tilde\By(\lambda_{\varepsilon}-1/6)^{-1}[\nu_l]ds_{\tilde\By}\\
&=\frac{3((\mu_l-\mu_0)\varepsilon_s^2+2(\mu_0-\mu_l)\varepsilon_s\gamma_l+2(\mu_l+2\mu_0)\varepsilon_0\gamma_l)}{(\varepsilon_s+2\varepsilon_0)(\mu_l+2\mu_0)(2\gamma_l+\varepsilon_s)} |\Omega| I,
\end{split}
\eeq
which proves that $\mathbf{P}_k$ is nonsingular.
\end{proof}
We remark that $\Omega$ is not necessary to be a ball to ensure the nonsingularity of the matrix $\mathbf{P}_l$. Indeed, one can also explicitly calculate $\mathbf{P}_l$ if $\Omega$ is an ellipsoid and show that $\mathbf{P}_l$ is nonsingular if the parameters $\mu_l$ and $\gamma_l$ are not quite special. One can find from \eqnref{eq:reviseadd0101} that if $\varepsilon_s=\varepsilon_0$ and $\Gs\neq 0$ then $\mathbf{P}_l$ is nonsingular, which is indicated in Lemma \ref{le:ax0000101}. Starting from now on and throughout the rest of the paper, we always assume that $\mathbf{P}_l$, $l=1, 2, \ldots l_0$ are nonsingular.

\subsection{Spherical harmonics expansion}

{
In Theorem~\ref{le:maggra01}, we derived the necessary asymptotic expansion for our subsequent inverse problem study. Furthermore, at a certain point, we shall need the the expansion of the steady geomagnetic field on the surface of the Earth with respect to the spherical harmonic functions. To that end, we present the following lemma
\begin{lem}\label{le:04}
Let $\bZ\in B_{R_0}$ be fixed, where $B_{R_0}$ stands for a ball of radius $R_0\in\mathbb{R}_+$. Let $\Bx\in \p B_{R_1}$ and suppose $R_0<R_1$. There holds the following asymptotic expansion
\beq\label{eq:leasm02}
\nabla \Gamma(\Bx-\bZ)=\sum_{n=0}^{\infty}\sum_{m=-n}^{n}
\frac{(n+1)Y_n^m(\hat\Bx)\hat\Bx-\nabla_SY_n^m(\hat\Bx)}{(2n+1)R_1^{n+2}}\overline{Y_n^m(\hat{\bZ})}\|\bZ\|^{n},
\eeq
where
$\hat\bZ=\bZ/\|\bZ\|$ and $\hat\Bx=\Bx/\|\Bx\|$. $Y_n^m$ is the spherical harmonics of order $m$ and degree $n$.
\end{lem}
\begin{proof}
Suppose that $\|\Bx\|>\|\By\|$, then there holds the following addition formula (cf. \cite{CK,Ned})
  \beq\label{eq:le04tmp02}
  \frac{1}{4\pi\|\Bx-\By\|}=\sum_{n=0}^\infty \sum_{m=-n}^{n} \frac{1}{2n+1}Y_n^m(\hat\Bx) \overline{Y_n^m(\hat\By)} \,\frac{\|\By\|^n}{\|\Bx\|^{n+1}}.
  \eeq
Since $\|\Bx\|=R_1$, by using the definition of surface gradient, one has
\beq\label{eq:le04tmp01}
\nabla \frac{Y_n^m(\hat\Bx)}{\|\Bx\|^{n+1}}=-((n+1)Y_n^m(\hat\Bx)\hat\Bx-\nabla_SY_n^m(\hat\Bx))R_1^{-(n+2)}.
\eeq
By substituting \eqnref{eq:le04tmp01} into the gradient of \eqnref{eq:le04tmp02}, one can obtain \eqnref{eq:leasm02}.

The proof is complete.
\end{proof}

By substituting \eqnref{eq:leasm02} into \eqnref{eq:lemag0101}, one can obtain the spherical harmonic expansion of the magnetic field.}

\section{Unique recovery results for magnetized anomalies}

We are in a position to present the main unique recovery results in identifying the magnetized anomalies. In what follows, we let $D_l^{(1)}$ and $D_l^{(2)}$, $l=1, 2, \ldots, l_0$, be two sets of magnetic anomalies, which satisfy \eqnref{eq:permeab02} with $\Bz_l$ replaced by $\Bz_l^{(1)}$ and $\Bz_l^{(2)}$, respectively. Correspondingly, the material parameters $\varepsilon_l$, $\Gs_l$, $\gamma_l$ and $\mu_l$ are replaced by $\varepsilon_l^{(1)}$, $\Gs_l^{(1)}$, $\gamma_l^{(1)}$, $\mu_l^{(1)}$ and $\varepsilon_l^{(2)}$, $\Gs_l^{(2)}$, $\gamma_l^{(2)}$, $\mu_l^{(2)}$, respectively, for $D_l^{(1)}$ and $D_l^{(2)}$, $l=1, 2, \ldots, l_0$. Let $(\bE_j,\bH_j)$, $j=1, 2$, be the solutions to \eqnref{eq:paradef02} and \eqnref{eq:pss02} with $D_l$ replaced by $D_l^{(1)}$ and $D_l^{(2)}$, respectively. Denote by $\mathbf{D}_l^{(1)}$, $\mathbf{D}_l^{(2)}$, $\mathbf{P}_l^{(1)}$ and $\mathbf{P}_l^{(2)}$ the polarization tensors for $D_l^{(1)}$ and $D_l^{(2)}$, respectively, $l=1, 2, \ldots, l_0$.

Let $\bH_1^{(0)}$ and $\bH_2^{(0)}$ be the leading terms of $\bH_1$ and $\bH_2$, respectively.
Then from \eqnref{eq:lemag0101}, there holds the following for $\Bx\in\RR^3\setminus\Sigma$,
\beq\label{eq:lemag010101}
\begin{split}
\bH_j^{(0)}(\Bx)=&
\bH_0^{(0)}(\Bx)+\delta^3\sum_{l=1}^{l_0}\Big(\nabla\big(\nabla\Gamma_0(\Bx-\Bz_l^{(j)})^T\mathbf{P}_l^{(j)}\bH_0^{(0)}(\Bz_l^{(j)})\big) +\Ocal(\delta^4), \quad j=1,2,
\end{split}
\eeq
\begin{lem}\label{le:ax01}
If there holds
\beq\label{eq:thmain01}
\nu\cdot\bH_1=\nu\cdot\bH_2\neq 0 \ \ \mbox{on} \ \ \Gamma,
\eeq
then one has
\beq\label{eq:leax02}
\sum_{m=-n}^n{\mathbf{N}_{n+1}^m(\hat\Bx)}^T\mathbf{d}_1^{n,m}=\sum_{m=-n}^n{\mathbf{N}_{n+1}^m(\hat\Bx)}^T\mathbf{d}_2^{n,m}, \quad \hat\Bx\in \mathbb{S}^2
\eeq
for any $n\in \mathbb{N}\cup\{0\}$, where
\beq\label{eq:leax03}
\mathbf{d}_j^{n,m}=\sum_{l=1}^{l_0}\overline{Y_n^m(\hat{\bZ}_l^{(j)})}\|\bZ_l^{(j)}\|^{n}(\mathbf{P}_l^{(j)})^T\bH_0^{(0)}(\bZ_l^{(j)}), \quad j=1,2,
\eeq
and
\beq\label{eq:leax04}
\mathbf{N}_{n+1}^{m}(\hat\Bx)=(n+1)Y_n^m(\hat\Bx)\hat\Bx-\nabla_SY_n^m(\hat\Bx).
\eeq
\end{lem}
\begin{proof}
First, by using \eqnref{eq:thmain01} and unique continuation, one sees that
$$\bH_1=\bH_2 \quad \mbox{in} \ \ \RR^3\setminus\overline{\Sigma},$$
Then from \eqnref{eq:lemag010101} one has
\beq\label{eq:leaddtmp0101}
\sum_{l=1}^{l_0}\big(\nabla^2\Gamma_0(\Bx-\Bz_l^{(1)})\mathbf{D}_l^{(1)}\bH_0^{(0)}(\Bz_l^{(1)})\big)=
\sum_{l=1}^{l_0}\big(\nabla^2\Gamma_0(\Bx-\Bz_l^{(2)})\mathbf{D}_l^{(2)}\bH_0^{(0)}(\Bz_l^{(2)})\big) \ \ \mbox{in} \ \ \RR^3\setminus\overline{\Sigma}.
\eeq
Suppose $R_1\in\mathbb{R}_+$ is sufficiently large such that $\Sigma\Subset B_{R_1}$. By \eqnref{eq:leaddtmp0101} one readily has
\beq\label{eq:leaxtmp01}
\begin{split}
\nu\cdot\nabla\sum_{l=1}^{l_0}\nabla\Gamma_0(\Bx-\Bz_l^{(1)})^T\mathbf{P}_l^{(1)}\bH_0^{(0)}(\Bz_l^{(1)})
=\nu\cdot\nabla\sum_{l=1}^{l_0}\nabla\Gamma_0(\Bx-\Bz_l^{(2)})^T\mathbf{P}_l^{(2)}\bH_0^{(0)}(\Bz_l^{(2)}), \quad \mbox{on} \, \p B_{R_1}.
\end{split}
\eeq
On the other hand, it can be verified that
$$u_j(\Bx):=\sum_{l=1}^{l_0}\nabla\Gamma_0(\Bx-\Bz_l^{(j)})^T\mathbf{P}_l^{(j)}\bH_0^{(0)}(\Bz_l^{(j)}), \quad j=1, 2$$
are harmonic functions in $\RR^3\setminus\overline{B_{R_1}}$, which decay at infinity. Using this together with \eqnref{eq:leaxtmp01}, and the maximum principle of harmonic functions, one can obtain that
$$
u_1(\Bx)=u_2(\Bx),\quad \Bx \in \RR^3\setminus\overline{B_{R_1}},
$$
and therefore
\beq\label{eq:leaxtmp02}
u_1(\Bx)=u_2(\Bx),\quad \Bx \in \p B_{R_1}.
\eeq
By substituting \eqnref{eq:leasm02} into \eqnref{eq:leaxtmp02} one has
\beq\label{eq:leaxtmp03}
\sum_{n=0}^{\infty}\sum_{m=-n}^{n}
\frac{{\mathbf{N}_{n+1}^m(\hat\Bx)}^T\mathbf{d}_1^{n,m}}{(2n+1)R_1^{n+2}}
=\sum_{n=0}^{\infty}\sum_{m=-n}^{n}
\frac{{\mathbf{N}_{n+1}^m(\hat\Bx)}^T\mathbf{d}_2^{n,m}}{(2n+1)R_1^{n+2}}.
\eeq
By taking $R_1$ sufficiently large and comparing the orders of $R_1^{-1}$ one has
\eqnref{eq:leax02}.

The proof is complete.
\end{proof}
\subsection{Uniqueness in recovering a single anomaly}

We present the uniqueness result in recovering a single anomaly.
\begin{thm}\label{th:01}
Suppose $l_0=1$ and $\Gs_1^{(1)}, \Gs_1^{(2)}\neq 0$. If there holds \eqnref{eq:thmain01}
then $\bZ_1^{(1)}=\bZ_1^{(2)}$ and $\mu_1^{(1)}=\mu_1^{(2)}$.
\end{thm}
\begin{proof}
By the unique continuation principle, one has from \eqnref{eq:thmain01} that $\bH_1=\bH_2$ in $\RR^3\setminus\overline{D_1^{(1)}\cup D_1^{(2)}}$. First, we note that $Y_0^0=\frac{1}{2\sqrt{\pi}}$, and then by letting $n=0$ and using \eqnref{eq:leax03} and \eqnref{eq:leax04}, $l_0=1$, we have
\beq\label{eq:mainthtmp01}
\mathbf{d}_j^{0,0}=\frac{1}{2\sqrt{\pi}}\mathbf{M}_1^{(j)}\bH_0^{(0)}(\bZ_1^{(j)}) \quad \mbox{and} \quad \mathbf{N}_{1}^{0}(\hat\Bx)=\frac{1}{2\sqrt{\pi}}\hat\Bx, \quad j=1,2.
\eeq
Hence by using \eqnref{eq:leax02} there holds
\beq\label{eq:mainthtmp02}
\hat\Bx^T\mathbf{P}_1^{(1)}\bH_0^{(0)}(\bZ_1^{(1)})=\hat\Bx^T\mathbf{P}_1^{(2)}\bH_0^{(0)}(\bZ_1^{(2)}), \quad \forall\hat\Bx\in \mathbb{S}^2,
\eeq
which readily implies that
\beq\label{eq:mainthtmp03}
\mathbf{P}_1^{(1)}\bH_0^{(0)}(\bZ_1^{(1)})=\mathbf{P}_1^{(2)}\bH_0^{(0)}(\bZ_1^{(2)}).
\eeq
In the following, we set $\mathbf{c}:=\mathbf{P}_1^{(1)}\bH_0^{(0)}(\bZ_1^{(1)})$. We claim that
$$\bH_0^{(0)}(\bZ_1^{(1)})\neq 0.$$ From \eqnref{eq:le0101} one has that $\bH_0^{(0)}=\nabla u_0$, where $u_0$ is a non-constant harmonic function. Hence, by the maximum principle of harmonic functions, $\nabla u_0(\Bz_1^{(1)})\neq 0$. This together with the assumption that $\mathbf{P}_1^{(1)}$ is a nonsingular matrix, one has $\mathbf{c}\neq 0$. Now we set $n=1$, and define
\beq\label{eq:mainthtmp04}
\mathbf{Y}_1(\hat\Bx):=(Y_1^{-1}(\hat\Bx), Y_1^0(\hat\Bx), Y_1^1(\hat\Bx))^T.
\eeq
Let $(\hat\Bx,\mathbf{e}_\phi,\mathbf{e}_\theta)$ be the triplet of the local orthogonal unit vectors, where $\phi$ and $\theta$ depend on $\hat\Bx$.
Note that
$$
\nabla_SY_n^m=\frac{1}{\sin\theta}\frac{\p Y_n^m(\hat\Bx)}{\p \phi}\mathbf{e}_\phi+\frac{\p Y_n^m(\hat\Bx)}{\p \theta}\mathbf{e}_\theta.
$$
By using \eqnref{eq:leax02} again there holds
\beq\label{eq:mainthtmp05}
\begin{split}
&\|\bZ_1^{(1)}\|\overline{\mathbf{Y}_1(\hat\bZ_1^{(1)})}^T\mathbf{Q}(\hat\Bx)(\hat\Bx, \mathbf{e}_\phi, \mathbf{e}_\theta)^T\mathbf{c} \\
=&\|\bZ_1^{(2)}\|\overline{\mathbf{Y}_1(\hat\bZ_1^{(2)})}^T\mathbf{Q}(\hat\Bx)(\hat\Bx, \mathbf{e}_\phi, \mathbf{e}_\theta)^T\mathbf{c},
\quad \hat\Bx\in \mathbb{S}^2,
\end{split}
\eeq
where $\mathbf{Q}(\hat\Bx)$ is a matrix defined by
\beq\label{eq:mainthtmp06}
\mathbf{Q}(\hat\Bx):=\left[
2\mathbf{Y}_1(\hat\Bx), -\frac{1}{\sin\theta}\frac{\p \mathbf{Y}_1(\hat\Bx)}{\p \phi}, -\frac{\p \mathbf{Y}_1(\hat\Bx)}{\p \theta}
\right].
\eeq
Straightforward calculations show that $\mathbf{Q}(\hat\Bx)$ is  nonsingular for any $\hat\Bx\in \mathbb{S}^2$. Since $\hat\Bx\in \mathbb{S}^2$ is arbitrarily given, \eqnref{eq:mainthtmp05} implies that
\beq\label{eq:mainthtmp07}
\|\bZ_1\|\overline{\mathbf{Y}_1(\hat\bZ_1)}^T
=\|\bZ_2\|\overline{\mathbf{Y}_1(\hat\bZ_2)}^T.
\eeq
Then by direct calculations, one has $\|\bZ_1\|=\|\bZ_2\|$ and $\hat\bZ_1=\hat\bZ_2$. We thus have
$D_1^{(1)}=D_1^{(2)}$. Hence, in the sequel, we let $D_1:=D_1^{(1)}=D_1^{(2)}$.
Clearly, $\bH_1=\bH_2$ in $\RR^3\setminus\overline{D_1}$.
Since $\Gs_1^{(1)},\Gs_1^{(2)}\neq 0$, from \eqnref{eq:th0101}, one can find that
\beq\label{eq:addre0101}
\begin{split}
\bH_j=\hat\bH_0^{(0)}-\frac{\mu_0}{\mu_1^{(j)}-\mu_0}\nabla\Scal_{D_1}^{0}(\Gl_{\mu_1^{(j)}} I-(\Kcal_{D_1}^0)^*)^{-1}[\nu_1\cdot\hat\bH_0^{(0)}]+\Ocal(\omega)
\quad \mbox{in} \quad \tilde{D},
\end{split}
\eeq
By using the jump formula one further has that
\beq\label{eq:newrevise010102}
\begin{split}
&\frac{1}{\mu_1^{(1)}-\mu_0}\Big(\frac{I}{2}+(\Kcal_{D_1}^0)^*\Big)(\Gl_{\mu_1^{(1)}} I-(\Kcal_{D_1}^0)^*)^{-1}[\nu_1\cdot\hat\bH_0^{(0)}]|_-\\
=&\frac{1}{\mu_1^{(2)}-\mu_0}\Big(\frac{I}{2}+(\Kcal_{D_1}^0)^*\Big)(\Gl_{\mu_1^{(2)}} I-(\Kcal_{D_1}^0)^*)^{-1}[\nu_1\cdot\hat\bH_0^{(0)}]|_- \quad \mbox{on} \quad \p D_1.
\end{split}
\eeq
Using the fact that $\frac{I}{2}+(\Kcal_{D_1}^0)^*$ is invertible on $L^2(\p D_1)$ and
some elementary calculations, one has from \eqnref{eq:newrevise010102} that
\beq\label{eq:newrevise010103}
(\mu_1^{(1)}-\mu_1^{(2)})\Big(\frac{I}{2}-(\Kcal_{D_1}^0)^*\Big)(\Gl_{\mu_1^{(1)}} I-(\Kcal_{D_1}^0)^*)^{-1}[\nu_1\cdot\hat\bH_0^{(0)}]=0 \quad \mbox{on} \quad \p D_1.
\eeq
Note that $\nu\cdot\hat\bH_0^{(0)}\in L_0^2(\p D_1)$, where $L_0^2(\p D_1)$ is a subset of $L^2(\p D_1)$ with zero average on $\p D_1$. Since $\nu_1\cdot\hat\bH_0^{(0)}\neq 0$ (otherwise by the unique continuation of harmonic functions one has $\bH_0^{(0)}=0$ in $\RR^3\setminus\overline\Sigma$, and this cannot be true), one has from \eqnref{eq:leH1repre01new} that $\nu_1\cdot\hat\bH_0^{(0)}\neq 0$  on $\p D_1$. Using this and the fact that $-\frac{I}{2}+(\Kcal_{D_1}^0)^*$ is invertible on $L_0^2(\p D_1)$, we finally have from \eqnref{eq:newrevise010103} that $\mu_1^{(1)}=\mu_1^{(2)}$.

The proof is complete.
\end{proof}

\subsection{Uniqueness in recovering multiple anomalies}
\begin{thm}\label{th:02}
If there holds \eqnref{eq:thmain01}
then $\bZ_l^{(1)}=\bZ_l^{(2)}$, $l=1, 2, \ldots, l_0$.
\end{thm}
\begin{proof}
With our earlier preparations, the proof follows from a similar argument to that of Theorem 7.8 in \cite{HK07:book}. In the following, we only sketch it. Using the formula \eqnref{eq:leaxtmp01} and similar analysis in the proof of Lemma \ref{le:ax01}, one can show that
\beq
\sum_{l=1}^{l_0}\Big(\nabla\Gamma_0(\Bx-\Bz_l^{(1)})^T\mathbf{P}_l^{(1)}\bH_0^{(0)}(\Bz_l^{(1)})
-\nabla\Gamma_0(\Bx-\Bz_l^{(2)})^T\mathbf{P}_l^{(2)}\bH_0^{(0)}(\Bz_l^{(2)})\Big)=0,
\eeq
holds in $\RR^3\setminus\overline{\Sigma}$. By straightforward calculations, one can further show that
\beq\label{eq:reviseadd01}
\begin{split}
F(\Bx):=&\sum_{l=1}^{l_0}\Big((\nabla\Gamma_0(\Bx-\Bz_l^{(1)})-\nabla\Gamma_0(\Bx-\Bz_l^{(2)}))^T\mathbf{P}_l^{(1)}\bH_0^{(0)}(\Bz_l^{(1)})\\
&-\nabla\Gamma_0(\Bx-\Bz_l^{(2)})^T(\mathbf{P}_l^{(2)}\bH_0^{(0)}(\Bz_l^{(2)})-\mathbf{P}_l^{(1)}\bH_0^{(0)}(\Bz_l^{(1)}))\Big)\\
=&\sum_{l=1}^{l_0}\Big(\big(\nabla^2\Gamma_0(\Bx-\Bz_l')(\Bz_l^{(1)}-\Bz_l^{(2)})\big)^T\mathbf{P}_l\bH_0^{(0)}(\Bz_l^{(1)})\\
&-\nabla\Gamma_0(\Bx-\Bz_l^{(2)})^T(\mathbf{P}_l^{(2)}\bH_0^{(0)}(\Bz_l^{(2)})-\mathbf{P}_l^{(1)}\bH_0^{(0)}(\Bz_l^{(1)}))\Big)=0
\end{split}
\eeq
holds in $\RR^3\setminus\overline{\Sigma}$, where $\Bz_l'=\Bz_l^{(1)}+t'\Bz_l^{(2)}$ with $t'\in (0, 1)$. Note that $F(\Bx)$ defined in \eqnref{eq:reviseadd01} is also harmonic in $\RR^3\setminus\bigcup_{l=1}^{l_0}(\Bz_l^{(1)}\cup \Bz_l^{(2)})$. By using the analytic continuation of harmonic functions, one thus has that $F(\Bx)\equiv 0$ in $\RR^3$. Define $F:=F_1+F_2$, where
$$
F_1(\Bx):=\sum_{l=1}^{l_0}\big(\nabla^2\Gamma_0(\Bx-\Bz_l')(\Bz_l^{(1)}-\Bz_l^{(2)})\big)^T\mathbf{P}_l^{(1)}\bH_0^{(0)}(\Bz_l^{(1)}),
$$
and
$$
F_2(\Bx):=-\nabla\Gamma_0(\Bx-\Bz_l^{(2)})^T(\mathbf{P}_l^{(2)}\bH_0^{(0)}(\Bz_l^{(2)})-\mathbf{P}_l^{(1)}\bH_0^{(0)}(\Bz_l^{(1)})).
$$
Then by comparing the types of poles of $F_1$ and $F_2$, one immediately finds that $F_1=0$ and $F_2=0$ in $\RR^3$.
Since $\bH_0^{(0)}(\Bx)$ does not vanish for $\Bx\in \RR^3\setminus\overline{\Sigma}$, then one has $\mathbf{P}_l^{(1)}\bH_0^{(0)}(\Bz_l^{(1)})\neq 0$. Hence we have
$$
\Bz_l^{(1)}-\Bz_l^{(2)}=0, \quad l=1, 2, \ldots, l_0.
$$

The proof is complete.
\end{proof}

\begin{rem}
We remark that for the recovery of multiple anomalies, we can only prove the uniqueness in identifying the positions of the anomalies. In principle, our arguments developed in this work can also be used to show the identification of the magnetic permeability of the anomalies as well, similar to the single anomaly case (cf. Theorem~\ref{th:01}). However, it would involve much more complicated analysis and we leave it for our future study.
\end{rem}

\section{Concluding remark}\label{sect:5}

In this paper, we develop a mathematical theory for the applied technology of identifying magnetized anomalies using geomagnetic monitoring. We provide the mathematical modelling as a type of nonlinear inverse problem and establish the global uniqueness in recovering the locations of multiple magnetized anomalies. For the case with a single anomaly, we show that one can also identify the magnetic permeability of the anomaly. Our mathematical arguments rely on the asymptotic analysis of the geomagnetic fields with respect to the wave frequency and the size of the anomalies. We mainly make use of the steady part in the difference of the geomagnetic fields monitored before and after the presence of the magnetized anomalies. One can expect that the technical condition in \eqref{eq:geom1}, requiring that the geomagnetic field should be monitored for all the time, can be relaxed to a finite time interval. In fact, in a forthcoming article, we not only develop an efficient numerical reconstruction scheme for the geomagnetic monitoring problem based on the theory in the current article, but also numerically verify that the monitoring can indeed be conducted within a finite time interval. Our study also opens up intriguing mathematical topics for further developments, including the identification of moving anomalies using the geomagnetic monitoring and the investigation of geomagnetic monitoring for different planets other than the Earth, such as the Sun.

\section*{Acknowledgment}
The work of Y. Deng was supported by NSF grant of China No. 11601528, NSF grant of Hunan No. 2017JJ3432 and No. 2018JJ3622, Innovation-Driven Project of Central South University, No. 2018CX041, Mathematics and Interdisciplinary Sciences Project of Central South University. The work of H. Liu was supported by the FRG and startup grants from Hong Kong Baptist University, Hong Kong RGC General Research Funds, 12302415 and 12302017.

\appendix
\section{Uniqueness of solution}\label{app:01}
We prove the uniqueness of a trivial solution to \eqnref{eq:leunique07}. Let $\bE$ be the solution to \eqnref{eq:leunique07}. Since
$$\nabla\times \bE=0, (\RR^3\setminus\overline{\Sigma})\cup \Sigma_s,$$
and noting that $\RR^3\setminus\overline\Sigma$ and $\Sigma_s$ are simply connected domains, one can find $u_1\in H_{\rm{loc}}^1(\RR^3\setminus\overline\Sigma)$ and $u_2\in H^1(\Sigma_s)$, such that
\beq\label{eq:app01}
\bE=\left\{
\begin{array}{ll}
\nabla u_1  &\ \mbox{in} \ \ \RR^3\setminus\overline\Sigma, \\
\nabla u_2 &\ \mbox{in} \ \ \Sigma_s.
\end{array}
\right.
\eeq
Furthermore, $\nabla\cdot\bE=0$ implies that $\Delta u_1=0$ and $\Delta u_2=0$. This together with the fact that $\bE=\Ocal(\|\Bx\|^{-2})$ as $\Bx\rightarrow \infty$, and the Helmholtz decomposition, readily implies that $u_1=\Ocal(\|\Bx\|^{-1})$. Then by using the transmission and boundary conditions in \eqnref{eq:leunique07}, and integration by parts, we have
\beq\label{eq:app02}
\begin{split}
&\int_{\RR^3\setminus \overline\Sigma}\varepsilon_0|\nabla u_1|^2+\int_{\Sigma_s}\varepsilon_s|\nabla u_2|^2\\
=&-\int_{\p \Sigma}\varepsilon_0\frac{\p u_1}{\p \tilde\nu}\Big|_+ u_1|_+ + \int_{\p \Sigma}\varepsilon_s\frac{\p u_2}{\p \tilde \nu}\Big|_- u_2|_-
-\int_{\p \Sigma_c}\varepsilon_s \frac{\p u_2}{\p \nu}\Big|_+ u_2|_+\\
=&C\int_{\p \Sigma}\varepsilon_s\frac{\p u_2}{\p \tilde \nu}\Big|_- =C\int_{\p \Sigma}\Big(\varepsilon_s\frac{\p u_2}{\p \tilde \nu}\Big|_-  - \varepsilon_s \frac{\p u_2}{\p \nu}\Big|_+ \Big)=0,
\end{split}
\eeq
where $C$ is a constant.
By \eqnref{eq:app02} one thus has $u_1=C_1$ and $u_2=C_2$, where $C_1$ and $C_2$ are constants. Hence, $\bE=0$.

\section{Integral representations in Lemma~\ref{le:rep01}}\label{app:02}
In this appendix, we present the proof of Lemma~\ref{le:rep01}, namely the solution to \eqnref{eq:paradef02} and \eqnref{eq:pss01} admits the integral representations \eqnref{eq:solrepre01} and \eqnref{eq:solrepre02}. Let $\nu_l$ be the exterior unit normal vector to $\p D_l$, $l=1, 2, \ldots, l_0$. By using the transmission conditions on $\p \Sigma$ and $\p D_l$, $l=1, 2, \ldots, l_0$, one can obtain the following equations
\beq\label{phi_psi}
\left\{
\begin{split}
&\left(- \mu_0 I+\mathcal{D}_{\Sigma,\Sigma}^{k_0, k_s}\right)[\Phi_0]
+\mathcal{P}_{\Sigma,\Sigma}^{k_0,k_s}[\Psi_0]
=\sum_{l'=1}^{l_0}(\mathcal{D}_{\Sigma, D_{l'}}^{k_s, k_0}[\Phi_{l'}]+\mathcal{P}_{\Sigma, D_{l'}}^{k_s, k_0}[\Psi_{l'}])
\quad \mbox{on} \ \p \Sigma,\\
&\omega^2\left(-\frac{\varepsilon_s+\varepsilon_0}{2} I+\mathcal{F}_{\Sigma,\Sigma}^{k_0, k_s}\right)[\Psi_0]
+\mathcal{P}_{\Sigma,\Sigma}^{k_0,k_s}[\Phi_0]
=\sum_{l'=1}^{l_0}\Big(\omega^2\mathcal{F}_{\Sigma,D_{l'}}^{k_s, k_0}[\Psi_{l'}]
+\mathcal{P}_{\Sigma, D_{l'}}^{k_s, k_0}[\Phi_{l'}]\Big)\quad \mbox{on} \ \p \Sigma,\\
&\left(\frac{\mu_l+\mu_0}{2}I+\mathcal{D}_{D_l, D_l}^{\varsigma_l,k_s}\right)[\Phi_l]
+\sum_{l'\neq l}^{l_0}\mathcal{D}_{D_l, D_{l'}}^{\varsigma_{l},k_s}[\Phi_{l'}]+\sum_{l'=1}^{l_0}\mathcal{P}_{D_l, D_{l'}}^{\varsigma_{l},k_s}[\Psi_{l'}]\\
=&\nu_l\times\hat\bE_0+\mathcal{D}_{D_l, \Sigma}^{k_s,\varsigma_l}[\Phi_0]+
\mathcal{P}_{D_l, \Sigma}^{k_s,\varsigma_l}[\Psi_0] \quad \mbox{on} \ \p D_l,\\
&\omega^2\left(\frac{\gamma_l+\varepsilon_s}{2}I+\mathcal{F}_{D_l, D_l}^{\varsigma_l,k_s}\right)[\Psi_l]
+\omega^2\sum_{l'\neq l}^{l_0}\mathcal{F}_{D_l, D_{l'}}^{\varsigma_{l},k_s}[\Psi_{l'}]+\sum_{l'=1}^{l_0}\mathcal{P}_{D_l,D_{l'}}^{\varsigma_{l},k_s}[\Phi_{l'}]\\
=&i\omega\nu_l\times\hat\bH_0+\omega^2\mathcal{F}_{D_l,\Sigma}^{k_s,\varsigma_l}[\Psi_0]+
\mathcal{P}_{D_l,\Sigma}^{k_s,\varsigma_l}[\Phi_0]\quad \mbox{on} \ \p D_l,
\end{split}
\right.
\eeq
with the operators $\mathcal{P}_{D,D'}^{k,k'}$, $\mathcal{D}_{D,D'}^{k,k'}$ and $\mathcal{F}_{D,D'}^{k,k'}$ defined by
\beq\label{app:opdef01}
\begin{split}
\mathcal{P}_{D,D'}^{k,k'}:&=\Lcal_{D,D'}^{k}-\Lcal_{D,D'}^{k'}, \\
\mathcal{D}_{D,D'}^{k,k'}:&=\mu^{(k)}\Mcal_{D,D'}^{k}-\mu^{(k')}\Mcal_{D,D'}^{k'},\\
\mathcal{F}_{D,D'}^{k,k'}:&=\gamma^{(k)}\Mcal_{D,D'}^{k}-\gamma^{(k')}\Mcal_{D,D'}^{k'},
\end{split}
\eeq
where $k, k'\in \{k_0, k_s, \varsigma_1, \varsigma_2, \ldots, \varsigma_{l_0}\}$, $D, D'\in \{\Sigma, D_1, D_2, \dots, D_{l_0}\}$ and $\mu^{(k)}$, $\gamma^{(k)}$ are respectively the parameters $\mu$ and $\gamma$, which are related to $k$. In other words,
\begin{align*}
&\mu^{(k_0)}=\mu^{(k_s)}=\mu_0, \quad \mu^{(\varsigma_l)}=\mu_l, \\
&\gamma^{(k_0)}=\varepsilon_0, \quad \gamma^{(k_s)}=\varepsilon_s, \quad \gamma^{(\varsigma_l)}=\gamma_l.
\end{align*}
In the sequel, we prove that \eqnref{phi_psi} is uniquely solvable when $\omega\in\mathbb{R}_+$ is sufficiently small. By asymptotic analysis (see \eqnref{eq:asymtmp01} and \eqnref{eq:asymtmp02}), one can find that
\beq\label{eq:app03}
\begin{split}
\mathcal{F}_{\Sigma,\Sigma}^{k_0,k_s}=(\varepsilon_0-\varepsilon_s)\Mcal_{\Sigma}^0+\Ocal(\omega^2), \quad
\mathcal{F}_{\Sigma,D_l}^{k_s,k_0}=(\varepsilon_s-\varepsilon_0)\Mcal_{\Sigma, D_l}^0+\Ocal(\omega^2),\\
\mathcal{P}_{\Sigma,D_l}^{k_s,k_0}=(k_s^2-k_0^2)\nu\times(\Acal_{D_l}^{0}+D^2\mathcal{B}_{D_l})\Big|_{\p \Sigma}+\Ocal(\omega^3), \quad \mathcal{D}_{D_l, D_l}^{\varsigma_l, k_s}=(\mu_l-\mu_0)\Mcal_{D_l}^{0}+\Ocal(\omega^2),\\
\mathcal{F}_{D_l, D_{l}}^{\varsigma_l, k_s}=(\gamma_l-\varepsilon_s)\Mcal_{D_l}^0+\Ocal(\omega),\quad
\mathcal{P}_{D_l, D_{l'}}^{\varsigma_{l}, k_s}=(\varsigma_{l}^2-k_s^2)\nu_l\times(\Acal_{D_{l'}}^{0}+D^2\mathcal{B}_{D_{l'}})\Big|_{\p D_l}+\Ocal(\omega^2),\\
\mathcal{F}_{D_l,\Sigma}^{k_s,\varsigma_l}=(\varepsilon_s-\gamma_l)\Mcal_{D_l,\Sigma}^{0}+\Ocal(\omega), \quad
\mathcal{P}_{D_l,\Sigma}^{k_s,\varsigma_l}=(k_s^2-\varsigma_l^2)\nu_l\times(\Acal_{\Sigma}^{0}+D^2\mathcal{B}_{\Sigma})\Big|_{\p D_l}+\Ocal(\omega^2).
\end{split}
\eeq
{In \eqnref{eq:app03}, we only present the asymptotic expansion of some of the operators involved in \eqnref{phi_psi}, while the asymptotic behavior of the other operators can be derived in a similar manner.}
Then the first equation in \eqnref{phi_psi} implies
\beq\label{eq:app04}
\|\Phi_0\|_{{\rm TH}({\rm div}, \p \Sigma)}=\Ocal\left(\omega^2\Big(\|\Psi_0\|_{{\rm TH}({\rm div}, \p \Sigma)}+\sum_{l'=1}^{l_0}\big(\|\Phi_{l'}\|_{{\rm TH}({\rm div}, \p D_{l'})}+\|\Psi_{l'}\|_{{\rm TH}({\rm div}, \p D_{l'})}\big)\Big)\right).
\eeq
By substituting \eqnref{eq:app03} and \eqnref{eq:app04} into the second equation of \eqnref{phi_psi}, one further has
\beq\label{eq:app05}
\begin{split}
\Psi_0=&-\Big(\Gl_{\varepsilon} I +\Mcal_{\Sigma}^0\Big)^{-1}\sum_{l'=1}^{l_0}\Big(\Mcal_{\Sigma, D_{l'}}^0[\Psi_{l'}]
+\mu_0\nu\times(\Acal_{D_{l'}}^0+D^2\mathcal{B}_{D_{l'}})[\Phi_{l'}]\Big)\\
&+\Ocal\left(\omega^2\Big(\sum_{l'=1}^{l_0}\big(\|\Phi_{l'}\|_{{\rm TH}({\rm div}, \p D_{l'})}+\|\Psi_{l'}\|_{{\rm TH}({\rm div}, \p D_{l'})}\big)\Big)\right) \quad \mbox{on} \ \ \p \Sigma.
\end{split}
\eeq
By substituting \eqnref{eq:app04} and \eqnref{eq:app05} into the third and fourth equations in \eqnref{phi_psi}, one obtains
\beq\label{eq:app06}
\begin{split}
&\Big(\Gl_{\mu_l} I +\Mcal_{D_l}^0\Big)[\Phi_l]+\sum_{l'\neq l}^{l_0}\Mcal_{D_l, D_{l'}}^0[\Phi_{l'}]\\
&+\frac{\varsigma_l^2}{\mu_{l}-\mu_0}\nu_l\times\Big(\big(\Acal_\Sigma^0+D^2\mathcal{B}_\Sigma\big)[\Psi_0]
+\sum_{l'=1}^{l_0}\big(\Acal_{D_{l'}}^0+D^2\mathcal{B}_{D_{l'}}\big)[\Psi_{l'}]\Big)\\
=&\frac{1}{\mu_l-\mu_0}\nu_l\times\hat\bE_0+\Ocal\Big(\sum_{l'=1}^{l_0}\big((\varsigma_{l}^3+\omega^2)\|\Psi_{l'}\|_{{\rm TH}({\rm div}, \p D_{l'})}+\omega^2\|\Phi_{l'}\|_{{\rm TH}({\rm div}, \p D_{l'})}\big)\Big) \quad \mbox{on} \ \ \p D_l,
\end{split}
\eeq
and
\beq\label{eq:app06add}
\begin{split}
&\Big(\Gl_{\gamma_l} I +\Mcal_{D_l}^0\Big)[\omega\Psi_l]+\sum_{l'\neq l}^{l_0}\Mcal_{D_l, D_{l'}}^0[\omega\Psi_{l'}]\\
&+\sum_{l'=1}^{l_0}\omega^{-1}\frac{\varsigma_l^2-k_s^2}{\gamma_l-\varepsilon_s}
\nu_l\times(\Acal_{D_{l'}}^0+D^2\mathcal{B}_{D_{l'}})[\Phi_{l'}]+\omega\Mcal_{D_l,\Sigma}^0[\Psi_0]\\
=&i\frac{1}{\gamma_l-\varepsilon_s}\nu_l\times\hat\bH_0+\Ocal\Big(\sum_{l'=1}^{l_0}\big(\omega\varsigma_{l}^2\|\Psi_{l'}\|_{{\rm TH}({\rm div}, \p D_{l'})}+\omega^{-1}\varsigma_{l}^3\|\Phi_{l'}\|_{{\rm TH}({\rm div}, \p D_{l'})}\big)\Big) \quad \mbox{on} \ \ \p D_l,
\end{split}
\eeq
for $l=1, 2, \ldots, l_0$, where $\Gl_{\mu_l}$ and $\Gl_{\gamma_l}$ are defined in \eqnref{eq:defmulam}. We claim that the following equations are uniquely solvable
\beq\label{eq:app07}
\begin{split}
&\Big(\Gl_{\mu_l} I +\Mcal_{D_l}^0\Big)[\Phi_l]+\sum_{l'=1}^{l_0}\Mcal_{D_l, D_{l'}}^0[\Phi_{l'}]=\frac{1}{\mu_l-\mu_0}\nu_l\times\hat\bE_0, \quad \mbox{on} \, \p D_l, \quad l=1, 2, \ldots, l_0,
\end{split}
\eeq
when $D_l$, $l=1, 2, \ldots, l_0$ are far away from each other. For relevant details, we refer to Section \ref{sec:02}. Denote by $\mathbb{M}_D$ the $l_0$-by-$l_0$ matrix type operator defined on ${\rm TH}({\rm div}, \p D_1)\times {\rm TH}({\rm div}, \p D_2)\times\cdots\times {\rm TH}({\rm div}, \p D_{l_0})$
\beq\label{eq:thpf010302}
\begin{split}
\mathbb{M}_D:=
\left(
\begin{array}{cccc}
\Mcal_{D_1}^{0} & \Mcal_{D_1,D_2}^{0} & \cdots & \Mcal_{D_1, D_{l_0}}^{0} \\
\Mcal_{D_2,D_1}^{0} & \Mcal_{D_2}^{0} & \cdots & \Mcal_{D_2,D_{l_0}}^{0}\\
\vdots & \vdots &\ddots &\vdots \\
\Mcal_{D_{l_0},D_1}^{0} & \Mcal_{D_{l_0},D_2}^{0} & \cdots & \Mcal_{D_{l_0}}^{0}
\end{array}
\right).
\end{split}
\eeq
Then the operator
\beq\label{eq:thpf010303}
\mathbb{L}_D^{\mu}:=\left(
\begin{array}{cccc}
\lambda_{\mu_1}I & 0 & \cdots & 0\\
0 & \lambda_{\mu_2}I & \cdots & 0\\
\vdots & \vdots &\ddots &\vdots \\
0 & 0 & \cdots & \lambda_{\mu_{l_0}}I
\end{array}
\right)+\mathbb{M}_D
\eeq
is invertible on ${\rm TH}({\rm div}, \p D_1)\times {\rm TH}({\rm div}, \p D_2)\times\cdots\times {\rm TH}({\rm div}, \p D_{l_0})$.
From \eqnref{eq:app06} one obtains that
\beq\label{eq:app08}
\begin{split}
\Phi_l=&\Big((\mathbb{L}_D^{\mu})^{-1}\big[(\frac{(\nu_1\times(\hat\bE_0+\mathbf{A}))^T}{\mu_1-\mu_0}, \frac{(\nu_2\times(\hat\bE_0+\mathbf{A}))^T}{\mu_2-\mu_0}, \ldots, \frac{(\nu_{l_0}\times(\hat\bE_0+\mathbf{A}))^T}{\mu_{l_0}-\mu_0})^T\big]\Big)\cdot(\mathbf{e}_l\otimes(1, 1, 1)^T)\\
&+\Ocal\Big(\sum_{l'=1}^{l_0}\big((\varsigma_{l}^3+\omega^2)\|\Psi_{l'}\|_{{\rm TH}({\rm div}, \p D_{l'})}+\omega^2\sum_{l'=1}^{l_0}\|\nu_{l'}\times\hat\bE_0\|_{{\rm TH}({\rm div}, \p D_{l'})}\Big), \quad l=1, 2, \ldots, l_0,
\end{split}
\eeq
where
\beq
\label{eq:appedixdefA}
\mathbf{A}:=\frac{\varsigma_l^2}{\mu_{l}-\mu_0}\Big(\big(\Acal_\Sigma^0+D^2\mathcal{B}_\Sigma\big)\Big(\Gl_{\varepsilon} I +\Mcal_{\Sigma}^0\Big)^{-1}\sum_{l'=1}^{l_0}\Mcal_{\Sigma, D_{l'}}^0[\Psi_{l'}]
-\sum_{l'=1}^{l_0}\big(\Acal_{D_{l'}}^0+D^2\mathcal{B}_{D_{l'}}\big)[\Psi_{l'}]\Big).
\eeq
Here, the notation $\otimes$ stands for the Kronecker product and $\mathbf{e}_l$ is the $l_0$-dimensional Euclidean unit vector. By taking the surface divergence of both sides of \eqnref{eq:app08}, one can further obtain
\beq\label{eq:app09}
\begin{split}
\nabla_{\p D_l}\cdot\Phi_l=&-i\omega\mu_0\Big((\mathbb{J}_D^{\mu})^{-1}\big[(\frac{\nu_1\cdot\hat\bH_0}{\mu_1-\mu_0}, \frac{\nu_2\cdot\hat\bH_0}{\mu_2-\mu_0}, \ldots, \frac{\nu_{l_0}\cdot\hat\bH_0}{\mu_{l_0}-\mu_0})^T\big]\Big)_l\\
&-\Big((\mathbb{J}_D^{\mu})^{-1}\big[(\frac{\varsigma_1^2\nu_1\cdot\mathbf{B}}{\mu_1-\mu_0}, \frac{\varsigma_2^2\nu_2\cdot\mathbf{B}}{\mu_2-\mu_0}, \ldots, \frac{\varsigma_{l_0}^2\nu_{l_0}\cdot\mathbf{B}}{\mu_{l_0}-\mu_0})^T\big]\Big)_l\\
&+\Ocal\Big(\sum_{l'=1}^{l_0}\big((\varsigma_{l}^3+\omega^2)\|\Psi_{l'}\|_{{\rm TH}({\rm div}, \p D_{l'})}+\omega^3\sum_{l'=1}^{l_0}\|\nu_{l'}\cdot\hat\bH_0\|_{L^2(\p D_{l'})}\Big), \quad l=1, 2, \ldots, l_0,
\end{split}
\eeq
with
\beq
\label{eq:appedixdefA}
\mathbf{B}:=\nabla\times\Acal_\Sigma^0\Big(\Gl_{\varepsilon} I +\Mcal_{\Sigma}^0\Big)^{-1}\sum_{l'=1}^{l_0}\Mcal_{\Sigma, D_{l'}}^0[\Psi_{l'}]
-\nabla\times\sum_{l'=1}^{l_0}\Acal_{D_{l'}}^0[\Psi_{l'}],
\eeq
and the operator $\mathbb{J}_D^{\mu}$ defined on $L^2(\p D_1)\times L^2(\p D_2)\times\cdots\times L^2(\p D_{l_0})$ given by
\beq\label{eq:app10}
\mathbb{J}_D^{\mu}:=\left(
\begin{array}{cccc}
\lambda_{\mu_1}I & 0 & \cdots & 0\\
0 & \lambda_{\mu_2}I & \cdots & 0\\
\vdots & \vdots &\ddots &\vdots \\
0 & 0 & \cdots & \lambda_{\mu_{l_0}}I
\end{array}
\right)-\mathbb{K}_D^*,
\eeq
where $\mathbb{K}_D^*$ is an $l_0$-by-$l_0$ matrix type operator defined on $L^2(\p D_1)\times L^2(\p D_2)\times\cdots\times L^2(\p D_{l_0})$
\beq\label{eq:app11}
\begin{split}
\mathbb{K}_D^*:=
\left(
\begin{array}{cccc}
(\Kcal_{D_1}^{0})^* & \nu_1\cdot\nabla\Scal_{D_2}^{0} & \cdots & \nu_1\cdot\nabla\Scal_{D_{l_0}}^{0} \\
\nu_2\cdot\nabla\Scal_{D_1}^{0} & (\Kcal_{D_2}^{0})^* & \cdots & \nu_2\cdot\nabla\Scal_{D_{l_0}}^{0}\\
\vdots & \vdots &\ddots &\vdots \\
\nu_{l_0}\cdot\nabla\Scal_{D_1}^{0} & \nu_{l_0}\cdot\nabla\Scal_{D_2}^{0} & \cdots & (\Kcal_{D_{l_0}}^{0})^*
\end{array}
\right).
\end{split}
\eeq
Finally, by substituting \eqnref{eq:app05}, \eqnref{eq:app08} and \eqnref{eq:app09} into \eqnref{eq:app06add}, one can have
\beq\label{eq:app12}
\begin{split}
&\Big(\Gl_{\gamma_l} I +\Mcal_{D_l}^0\Big)[\omega\Psi_l]+\sum_{l'\neq l}^{l_0}\Mcal_{D_l, D_{l'}}^0[\omega\Psi_{l'}]-\Mcal_{D_l,\Sigma}^0\Big(\Gl_{\varepsilon} I +\Mcal_{\Sigma}^0\Big)^{-1}\sum_{l'=1}^{l_0}\Mcal_{\Sigma, D_{l'}}^0[\omega\Psi_{l'}]\\
&=i\frac{1}{\gamma_l-\varepsilon_s}\nu_l\times\hat\bH_0+\Ocal\Big(\sum_{l'=1}^{l_0}\big(\omega\varsigma_{l}^2\|\Psi_{l'}\|_{{\rm TH}({\rm div}, \p D_{l'})}+\omega^{-1}\varsigma_{l}^3\sum_{l'=1}^{l_0}\|\nu_{l'}\times\hat\bE_0\|_{{\rm TH}({\rm div}, \p D_{l'})}\big)\Big).
\end{split}
\eeq
Similarly, one can prove that \eqnref{eq:app12} is uniquely solvable (see Section \ref{sec:02}). Define
\beq\label{eq:app14}
\begin{split}
\mathbb{N}_D:=&{\rm diag}\Big(\big(\Mcal_{D_1,\Sigma}, \Mcal_{D_2,\Sigma}, \ldots, \Mcal_{D_{l_0},\Sigma}\big)\big(\Gl_{\varepsilon} I +\Mcal_{\Sigma}^0\big)^{-1}\Big)\\
&\quad\quad\quad\left(
\begin{array}{cccc}
\Mcal_{\Sigma, D_1}^{0} & \Mcal_{\Sigma,D_2}^{0} & \cdots & \Mcal_{\Sigma, D_{l_0}}^{0} \\
\Mcal_{\Sigma,D_1}^{0} & \Mcal_{\Sigma, D_2}^{0} & \cdots & \Mcal_{\Sigma,D_{l_0}}^{0}\\
\vdots & \vdots &\ddots &\vdots \\
\Mcal_{\Sigma,D_1}^{0} & \Mcal_{\Sigma,D_2}^{0} & \cdots & \Mcal_{\Sigma, D_{l_0}}^{0}
\end{array}
\right),
\end{split}
\eeq
and
\beq\label{eq:app15}
\mathbb{L}_D^{\gamma}:=\left(
\begin{array}{cccc}
\lambda_{\gamma_1}I & 0 & \cdots & 0\\
0 & \lambda_{\gamma_2}I & \cdots & 0\\
\vdots & \vdots &\ddots &\vdots \\
0 & 0 & \cdots & \lambda_{\gamma_{l_0}}I
\end{array}
\right)+\mathbb{M}_D-\mathbb{N}_D.
\eeq
One can show that
\beq\label{eq:app16}
\omega\Psi_l=i\Psi_l^{(0)}+\Ocal(\omega), \quad l=1, 2, \ldots, l_0,
\eeq
where $\Psi_l^{(0)}$ satisfies
\beq\label{eq:app17}
\begin{split}
\Psi_l^{(0)}=
&\Big((\mathbb{L}_D^{\gamma})^{-1}\big[(\frac{(\nu_1\times\hat\bH_0)^T}{\gamma_1-\varepsilon_s}, \frac{(\nu_2\times\hat\bH_0)^T}{\gamma_2-\varepsilon_s}, \ldots, \frac{(\nu_{l_0}\times\hat\bH_0)^T}{\gamma_{l_0}-\varepsilon_s})^T\big]\Big)\cdot(\mathbf{e}_l\otimes(1, 1, 1)^T).
\end{split}
\eeq

\end{document}